\documentclass[a4paper,fleqn]{cas-sc}

\usepackage[numbers]{natbib}
\def\tsc#1{\csdef{#1}{\textsc{\lowercase{#1}}\xspace}}
\tsc{WGM}
\tsc{QE}

\usepackage{amssymb}
\usepackage{amsmath}
\usepackage{amsthm}
\usepackage{mathtools}
\usepackage{siunitx}
\usepackage{enumitem}
\usepackage{bm}
\usepackage{epsfig}
\usepackage{tikz}
\usepackage{caption, subcaption}
\usepackage{float}
\allowdisplaybreaks

\newcommand{\rmd}{\mathrm{d}}
\newcommand{\bfx}{\bm{x}}
\newcommand{\bfq}{\bm{q}}
\newcommand{\bff}{\bm{f}}
\newcommand{\bfu}{\bm{u}}
\newcommand{\bfy}{\bm{y}}
\newcommand{\bfh}{\bm{h}}
\newcommand{\bfz}{\bm{z}}
\newcommand{\R}{\mathbb{R}}
\newcommand{\bfphi}{\bm{\Phi}}
\newcommand{\bfpsi}{\bm{\Psi}}
\newcommand{\opsi}{\overline{\bm{\Psi}}}

\newcommand{\G}{\mathcal{G}}
\newcommand{\A}{\bm{A}}
\renewcommand{\H}{\bm{H}}
\newcommand{\B}{\bm{B}}
\newcommand{\M}{\mathcal{M}}

\newcommand{\W}{\bm{W}}
\DeclareMathOperator{\tr}{Tr}
\newcommand{\lam}{\bm{\lambda}}
\newcommand{\bfe}{\bm{e}}
\newcommand{\C}{\bm{C}}
\newcommand{\K}{\bm{K}}
\newcommand{\Q}{\bm{Q}}
\newcommand{\MR}{\bm{R}}
\renewcommand{\S}{\bm{S}}

\newcommand{\PP}{\mathbb{P}}
\newcommand{\Z}{\bm{Z}}
\newcommand{\U}{\bm{U}}
\newcommand{\E}{\bm{E}}
\newcommand{\bfv}{\bm{v}}
\DeclareMathOperator{\mat}{Mat}
\DeclarePairedDelimiterX{\inp}[2]{\langle}{\rangle}{#1, #2}

\DeclareGraphicsExtensions{.pdf}

\makeatletter
\def\thmheadbrackets#1#2#3{%
  \thmname{#1}\thmnumber{\@ifnotempty{#1}{ }\@upn{#2}}%
  \thmnote{ {\the\thm@notefont[#3]}}}
\makeatother

\newtheoremstyle{brakets}% Name
  {}% space above
  {}% space below
  {\itshape}% body font
  {}% indent
  {\bfseries}% head font
  {.}% punctuation after head
  { }% space after head (has to be space or dimension!)
  {\thmheadbrackets{#1}{#2}{#3}}% head spec

\theoremstyle{brakets}
\newtheorem{prop}{Proposition}[section]

\begin{document}
\let\WriteBookmarks\relax
\def\floatpagepagefraction{1}
\def\textpagefraction{.001}

\shorttitle{Guaranteed-stable Non-intrusive Optimization of Reduced-Order Models}
\shortauthors{C.J. Errico et~al.}

\title[mode=title]{GasNiTROM: Model Reduction via Non-Intrusive Optimization of Oblique Projection Operators and Guaranteed-Stable Latent-Space Dynamics}

\author[1]{Cole J. Errico}[type=editor,orcid=0000-0002-3014-7498]
\cormark[1]
\ead{cerrico2@illinois.edu}
\author[2]{Alberto Padovan}
\ead{alberto.padovan@njit.edu}
\author[1]{Daniel J. Bodony}
\ead{bodony@illinois.edu}

\affiliation[1]{organization={Department of Aerospace Engineering, University of Illinois Urbana-Champaign},%Department and Organization
            city={Urbana},
            state={Illinois},
            postcode={61801},
            country={USA}}
\affiliation[2]{organization={Department of Mechanical and Industrial Engineering, New Jersey Institute of Technology},
            city={Newark},
            state={New Jersey},
            postcode={07102},
            country={USA}}

\cortext[cor1]{Corresponding author}

%% Abstract
\begin{abstract}
Non-intrusive reduced-order modeling techniques are necessary for systems that are simulated using black-box solvers or known only from data. For systems exhibiting large transients and operating far away from equilibria, current non-intrusive models often exhibit poor forecasting accuracy and can even be unstable in infinite or finite time. Recent developments have addressed the stability issue by seeking structure-preserving latent-space architectures when reducing Hamiltonian or Lagrangian full-order dynamics, or by enforcing global stability via Lyapunov-informed parameterizations in the latent space. However, such developments do not necessarily improve the forecasting accuracy of the resulting models, since these formulations achieve dimensionality reduction using orthogonal projections that accidentally truncate dynamically-important states. In this paper, we address both issues by introducing a non-intrusive framework designed to simultaneously identify globally-asymptotically-stable latent-space dynamics, and oblique projection operators capable of capturing the sensitivity mechanisms of the system. In particular, given a Lyapunov-based parameterization of the latent-space tensors, and a matrix-manifold parameterization of the oblique projection operators, we fit a model against high-fidelity training trajectories. Furthermore, we show that the gradient of the objective function can be written in closed form using adjoint-based backpropagation in the latent space, eliminating the need for automatic differentiation. We compare our formulation with state-of-the-art methods on a three-dimensional system of ordinary differential equations, and a two-dimensional lid-driven cavity flow at Reynolds number $Re=8300$. We demonstrate that our models are not only globally asymptotically stable (as expected by construction), but they are also significantly more accurate.
\end{abstract}

% Intrusive Petrov-Galerkin models have been shown to attain a high accuracy, but these formulations require access to the governing equations and potentially their adjoint, which is difficult or impossible to obtain from commercial codes.
% Recent advances to non-intrusive approaches such as Operator Inference have alleviated the stability issue, but fail to accurately capture the dynamics due to orthogonal projection of the full-order data.

%%Research highlights
% \begin{highlights}
% \item We introduce a novel non-intrusive model reduction method where we simultaneously seek globally-asymptotically-stable latent-space dynamics and oblique-projection operators capable of predicting the dynamics of nonlinear systems along transient trajectories.
% \item The latent-space dynamics and the projections are parameterized over an appropriate matrix manifold defined as the product of a Grassmannian, a Stiefel manifold, and several linear spaces.
% \item The gradient of the cost function with respect to the parameters is derived in closed-form using adjoint-based backpropagation, thereby bypassing the need for expensive automatic differentiation.
% \item Via examples, including a two-dimensional fluid flow, we demonstrate improved accuracy and robustness over existing state-of-the-art methods.
% \end{highlights}

%% Keywords
\begin{keywords}
model order reduction \sep data-driven reduced-order models \sep globally-stable reduced-order models \sep oblique projection \sep Operator Inference \sep NiTROM
\MSC 34D20 \sep 34D23 \sep 37M05 \sep 37M10 \sep 37N10
\end{keywords}

\maketitle

%% main text
%%

\section{Introduction}

Low-dimensional models of real-world processes enable important scientific tasks, including the rapid prototyping of engineering devices, development of control strategies, and acceleration of otherwise computationally expensive simulations. 
Many high-dimensional systems across physics and engineering exhibit seemingly complex behavior that can nonetheless be explained by much simpler low-dimensional models \cite{mohammed-taifourUnsteadinessLargeTurbulent2016,holmes2012turbulence}.
% The goal of these ``reduced-order models'' is to identify these structures and reduce the computational expense in predicting responses of the governing equations, for use in estimation, control, or generating additional insights.
These models usually contain the following components: a map from the high-dimensional state space to a low-dimensional space (i.e., an encoder), a map from the low-dimensional space to the original high-dimensional space (i.e., a decoder), and possibly nonlinear reduced-order dynamics that evolve the reduced-order state in the latent space.
The choice of these components fully determines the predictive accuracy and temporal stability of the resulting models, and must therefore be chosen carefully.
Here, we provide a brief review of methods where the encoder and decoder define linear projection operators, and we then review recent developments to strongly enforce global stability of the latent-space dynamics.
% We then provide an overview on guaranteeing stability of latent-space dynamics, and recent work in coupling the two concepts.

% There exist several techniques to obtain such models, but here, we narrow our focus to that of projection-based model order reduction \cite{bennerSurveyProjectionBasedModel2015}. 
% The most common category of reduced-order models that rely on linear projection operators are known as Petrov-Galerkin models.
Petrov-Galerkin models are perhaps the most well-known models relying on linear projection operators.
Given a decoder $\bfphi (\bfpsi^\top \bfphi)^{-1}$ and an encoder $\bfpsi^\top$, where $\bfphi$ and $\bfpsi$ are tall rectangular matrices, the aforementioned projection is given by $\PP = \bfphi(\bfpsi^\top \bfphi)^{-1} \bfpsi^\top$.
% A graphical example of the utility of these operators is shown in Figure 1 of \cite{rowley_model_2017}, where projections spanning two-dimensional planes approximate dynamics evolving in a three-dimensional space.
If $\bfphi = \bfpsi$, then the projection $\PP$ is orthogonal and the resulting model is known as a Galerkin model.
A very common approach to identify a candidate subspace for Galerkin models is Proper Orthogonal Decomposition (POD), first applied to the study of complex fluid flows by Lumley \cite{lumley1967structure,berkoozProperOrthogonalDecomposition1993}.
Despite the widespread use of POD-Galerkin models \cite{rowley_model_2004,baroneStableGalerkinReduced2009}, it is well-known that these often fail to be sufficiently accurate in systems exhibiting large-amplitude transient growth and high sensitivity to low-amplitude states or disturbances.
This is due to the fact that orthogonal projections onto high-energy/high-variance subspaces often truncate the dynamically-important mechanisms that are necessary to explain/predict the dynamics of the full-order system.
% The disadvantage of these POD-Galerkin models is exhibited in systems that exhibit large-amplitude transient growth.
% These systems are characterized by their non-normal linear dynamics, which require the use of carefully chosen oblique projection operators.
In linear systems, or nonlinear systems evolving in the proximity of an equilibrium, methods such as Balanced Truncation \cite{moore2003principal,willcox2002balanced} or Balanced POD \cite{rowley2005model} address this issue by computing oblique projections and corresponding Petrov-Galerkin models by balancing the observability and controllability Gramians associated with the underlying linear state space model.
Beyond balancing, reduced-order models can be obtained by minimizing the $\mathcal{H}_2$ and $\mathcal{H}_\infty$ norms of the error between the full-order and reduced-order transfer functions of the associated linear system \cite{antoulas2020interpolatory,van2008h2}.
For nonlinear systems evolving far away from an attractor, methods such as Least-Squares Petrov-Galerkin (LSPG) \cite{carlberg2011efficient,carlberg2017galerkin} can be used to compute an optimal decoder by solving a least-squares problem in the reduced-order space at each time step via an encoder specified by POD.
Recently developed data-driven methods such as Trajectory-based Optimization of Oblique Projections (TrOOP) \cite{otto_optimizing_2022} and Covariance Balancing Reduction using Adjoint Snapshots (CoBRAS) \cite{ottoModelReductionNonlinear2023,Zanardi2025,Zanardi2026} compute optimal (or quasi-optimal) oblique projection operators for Petrov-Galerkin modeling.
In particular, TrOOP identifies oblique projections by optimizing against trajectories from the full-order model, while CoBRAS identifies oblique projections by balancing the state and gradient covariance matrices associated with the full-order solution map.
By definition, Petrov-Galerkin formulations are intrusive, as they require explicit access to the full-order dynamics for model evaluation. 
Additionally, modeling formulations capable of computing oblique projection (e.g., balanced truncation, LSPG, TrOOP, etc.) require access to the linearization (and its adjoint) of the nonlinear dynamics evaluated around (possibly time-varying) base flows.
This significantly limits the applicability of these methods to real-world problems where the physics are simulated using commercial or legacy solvers with proprietary source codes nor are applicable to systems known only through available data.
% These can be difficult or impossible to extract from general black-box solvers, such as in many commercial codes.
Furthermore, these methods are not necessarily guaranteed to produce globally-stable reduced-order models, and are therefore susceptible to finite-time blow-up.
% Furthermore, these methods are not guaranteed to produce stable ROMs (aside from the linear modeling techniques mentioned above \cite{sandberg2004balanced,kavranoglu1993characterization,pernebo1982model} or POD-Galerkin under an appropriate inner product \cite{rowley_model_2004, baroneStableGalerkinReduced2009}), which could lead to solution blow-up.

A popular method to obtain non-intrusive reduced-order models on linear subspaces is Operator Inference \cite{peherstorfer_data-driven_2016,kramer_learning_2024}.
Operator Inference fits a model of polynomial form (although non-polynomial dynamics have also been considered \cite{bennerOperatorInferenceNonintrusive2020}) by minimizing the error between time-rate of change of latent space vectors and the orthogonally-projected full-order time-derivatives along training trajectories.
This task can be turned into a linear least-squares optimization problem whose solution can be computed in closed-form, thus making Operator Inference a very convenient model reduction formulation.
In its original form, Operator Inference does not guarantee stability of the latent space dynamics, and there have been several recent developments to address this issue.
For example, recent work \cite{goyal_guaranteed_2025, gkimisisRepresentationEnergypreservingQuadratic2026} proposes a Lyapunov-stability-inspired parameterization of the latent-space dynamics, thus guaranteeing global asymptotic stability by construction.
% We shall see that this approach is related to our formulation, and we refer to it henceforth as Globally-asymptotically-stable Operator Inference (GasOpInf).
In a similar fashion, \cite{gruber2025variationally,sharma2022hamiltonian,sharma2024lagrangian,issan2023predicting,shuai2025} extended the Operator Inference framework to leverage known continuous symmetries (e.g., translation invariance) and/or enforce Hamiltonian and Lagrangian structure in the latent space.
While these methods solve the stability issue at its root, they all achieve dimensionality reduction via orthogonal projection. 
As previously discussed, this can accidentally truncate low-energy, yet dynamically-important, states, ultimately leading to models with poor predictive accuracy.
% One of the main drawbacks associated with all of these methods is due to the orthogonal projection of the data, which produce poor predictions due to truncating out the dynamically-important states necessary for prediction of the full-order system.
The recently-developed ``Non-intrusive Trajectory-based optimization of Reduced-Order Models'' (NiTROM) \cite{padovan_data-driven_2024} approach addresses this by solving an optimization problem where the latent-space dynamics \emph{and} the (oblique) projection operators are computed simultaneously.
While this procedure is more computationally intensive than Operator Inference, treating the projection as an optimization parameter allows for the reduced model to detect the sensitivity mechanisms of the full-order governing equations, thus significantly improving its forecasting accuracy.
However, in its original form, NiTROM does not enforce any type of structure or stability constraints on the latent-space dynamics, and the resulting models can be susceptible to finite-time blow up.
% This has been shown to improve predictions of systems exhibiting large-amplitude transient growth, but does not enforce any form of stability on the latent-space dynamics.

In this paper, we build upon the NiTROM formulation \cite{padovan_data-driven_2024} and the recent work in \cite{goyal_guaranteed_2025}, and we propose a novel non-intrusive framework that simultaneously addresses the aforementioned orthogonal projection and stability problems.
In particular, given training trajectories from a full-order model, we fit a reduced-order model by simultaneously seeking Lyapunov-constrained reduced-order dynamics, and oblique projection operators defined by a linear encoder $\bfpsi^\top$ and linear decoder $\bfphi (\bfpsi^\top \bfphi)^{-1}$.
% Our approach is enabled by considering the subspace $V = \range(\bfphi)$, which is naturally an element of the Grassmann manifold, and by taking the matrix $\bfpsi$ to be an element of the orthogonal Stiefel manifold.
% Our approach is enabled by treating each parameter as an element of an associated Riemannian manifold, and then optimizing over their combined product manifold.
We shall see that the model can be parameterized over a Riemannian product manifold composed of a Grassmannian, a Stiefel manifold, and as many linear spaces as the tensors required to enforce Lyapunov stability on the latent-space dynamics.
% Furthermore, we optimize over the parameterized tensors of the linear and quadratic latent-space dynamics as in \cite{goyal_guaranteed_2025} such that the trajectories resulting from the model are guaranteed to be globally asymptotically stable in a Lyapunov sense.
Despite the apparent complexity, we also show that the gradient of the objective function with respect to the parameters can be written in closed form via adjoint-based backpropagation in the latent space, bypassing the need for automatic differentiation during the training process.
We test our formulation on two different examples: a simple three-dimensional system of ordinary differential equations, and the two-dimensional incompressible lid-driven cavity flow at Reynolds number $Re=8300$.
In these examples, we compare our framework with POD-Galerkin, Operator Inference, the globally-stable Operator Inference formulation of \cite{goyal_guaranteed_2025}, and NiTROM. 
For these examples, our models demonstrate better predictive accuracy than orthogonal projection-based models, and even outperform the NiTROM models, which exhibit finite-time blow-up in some cases.

\section{Mathematical formulation}
We start with a general nonlinear system with dynamics governed by
\begin{equation} \label{eq:goveq}
    \begin{aligned}
        \frac{\rmd \bfx}{\rmd t} &= \bff(\bfx, \bfu), \quad \bfx(0) = \bfx_0 \\
        \bfy &= \bfh(\bfx),
    \end{aligned}
\end{equation}
where $\bfx \in \R^n$ is the high-dimensional state vector, $\bfu \in \R^m$ is the external forcing input, and $\bfy \in \R^p$ is the measured output.
Henceforth, this system is referred to as the full-order model (FOM).
Our objective is to identify a reduced-order model (ROM) that can accurately predict the time-history of the output $\bfy$ in response to varying inputs $\bfu$ and initial conditions $\bfx_0$ at reduced computational cost.
Throughout, we consider ROMs of the form
\begin{equation} \label{eq:nitromsys}
    G(\bfphi, \bfpsi, \bff_r) = \begin{cases} \begin{aligned}
        \frac{\rmd \bfz}{\rmd t} &= \bff_r (\bfz, \bfu), \quad \bfz(0) = \bfpsi^\top \bfx_0, \\
        \hat{\bfy} &= \bfh \left(\bfphi (\bfpsi^\top \bfphi)^{-1} \bfz\right),
    \end{aligned}
    \end{cases}
\end{equation}
where $\bfz\in\mathbb{R}^r$ is the low-dimensional state vector (with $r\ll n$), $\bff_r:\mathbb{R}^r\times \mathbb{R}^m\to\mathbb{R}^r$ defines the latent-space (or reduced-order) dynamics, $\mathcal{D} = \bfphi (\bfpsi^\top \bfphi)^{-1}: \mathbb{R}^r\to\mathbb{R}^n$ is the so-called decoder, while $\mathcal{E}=\bfpsi^\top:\mathbb{R}^n\to\mathbb{R}^r$ is the encoder.
We emphasize that $\mathcal{E}$ is a left-inverse of $\mathcal{D}$ (i.e., $\mathcal{E}\mathcal{D}=I\in\bm{R}^{r\times r}$), so that the encoder/decoder pair defines a linear (oblique) projection operator $\mathbb{P}\coloneqq \mathcal{D}\mathcal{E}$ by construction.
% We also note in passing that the ROM $G$ in equation \eqref{eq:nitromsys} becomes a so-called Petrov-Galerkin model if 
% \begin{equation}
%     \bff_r(\bfz,\bfu) = \mathcal{E} \bff\left(\mathcal{D}\bfz,\bfu\right).
% \end{equation}
Clearly, the ROM $G$ in \eqref{eq:nitromsys} is fully defined by the parameters $\bfphi$, $\bfpsi$, and $\bff_r$, whose choice determines the predictive accuracy of the model as well as its stability and structural properties.
The primary concern of this paper is to identify \emph{non-intrusive} ROMs with \emph{dynamically-optimal} encoder/decoder pairs, and \emph{globally-asymptotically-stable} latent-space dynamics.
We build up towards this objective in the following subsections.

\subsection{Globally-asymptotically-stable latent-space dynamics} 
\label{sec:formulation}
In this subsection, we review recent work by \cite{gillis_computing_2017} and \cite{goyal_guaranteed_2025} to guarantee global stability via a careful Lyapunov-based parameterization of the latent-space dynamics $\bff_r$.
First, we recall that a system of the form \eqref{eq:nitromsys} (with $\bfu = 0$) is said to be Lyapunov-stable (or globally stable in the sense of Lyapunov) if there exists a positive function $V(\bfz)$ satisfying $dV(\bfz)/dt \leq 0$ for all $\bfz$.
The system is said to be globally \emph{asymptotically} stable if the condition holds with a strict inequality.
If we consider Lyapunov functions of the form $V(\bfz) = \bfz^\top \widetilde{\Q}\bfz$ with $\widetilde{\Q}$ a symmetric positive-definite matrix of appropriate size, then it can be shown straightforwardly that the system is Lyapunov-stable if and only if $C \coloneqq \bfz^\top \widetilde{\Q} \bff_r(\bfz) +  \bff_r(\bfz)^\top \widetilde{\Q}\bfz \leq 0$.

In order to strongly enforce this inequality constraint within a data-driven optimization framework, we assume polynomial (specifically, quadratic) dynamics of the form
\begin{equation}
\label{eq:quadsys}
    \bff_r(\bfz,\bfu) = \A\bfz + \H : \bfz\bfz^\top + \bm{B}\bfu,
\end{equation}
where $\A$, $\H$, and $\B$ are reduced-order tensors of appropriate dimension.
Following \cite{goyal_guaranteed_2025}, the strategy to strongly enforce global asymptotic stability is to identify tensors $\A$ and $\H$ satisfying
\begin{equation}
\label{eq:lyap_inequality}
    \begin{aligned}
        C = \underbrace{\bfz^\top \left(\widetilde{\Q}\A+\A^\top\widetilde{\Q}\right) \bfz}_{\coloneqq C_A < 0}
        +  \underbrace{\bfz^\top \widetilde{\Q} \H : \bfz\bfz^\top + \left(\H : \bfz\bfz^\top\right)^\top\widetilde{\Q} \bfz}_{\coloneqq C_H = 0} < 0\quad\forall\bfz.
    \end{aligned}
\end{equation}
In words, the linear dynamics defined by $\A$ are responsible for the asymptotic decay of the state $\bfz$ towards the origin, while the quadratic dynamics defined by $\H$ merely satisfy an energy-conservation equality.
In fact, we may say that the quadratic dynamics are energy-conserving, and they neither introduce nor remove energy from the system.
We emphasize that this strategy can be extended straightforwardly to higher-order polynomial dynamics by requiring that all higher-order tensors governing the cubic, quartic, etc. dynamics satisfy an analogous energy-conservation constraint.

At this point, in order to parameterize the tensors $\A$ and $\H$ to satisfy inequality~\eqref{eq:lyap_inequality} by construction, we invoke two results from \cite{gillis_computing_2017} and \cite{goyal_guaranteed_2025}.
The first result, reported in Proposition \ref{prop:stableA}, identifies a convenient parameterization for all (stable) matrices $\A$ satisfying the inequality $C_A$ in \eqref{eq:lyap_inequality}.
The second result in Proposition \ref{prop:stableH}, on the other hand, determines that all energy-conserving tensors $\H$ satisfying the equality $C_H$ in \eqref{eq:lyap_inequality} can be parameterized using skew-symmetric frontal slices.

\begin{prop}[Lemma 2 in \cite{gillis_computing_2017}]
    \label{prop:stableA}
    A matrix $\A \in\mathbb{R}^{r\times r}$ is Hurwitz (i.e., stable) if and only if it can be written as
    \begin{equation} \label{eq:linstable}
        \A = (\widetilde{\K} - \widetilde{\MR})\widetilde{\Q}
    \end{equation}
    for some skew-symmetric matrix $\widetilde{\K}$, positive semi-definite matrix $\widetilde{\MR}$, and positive-definite matrix $\widetilde{\Q}$.
\end{prop}

\begin{prop}[Corollary 4.1, Lemma 4.2 in \cite{goyal_guaranteed_2025}] \label{prop:stableH}
    Consider the third-order tensor $\H \in \R^{r \times r \times r}$ in equation \eqref{eq:quadsys}, and let $\widetilde{\Q}$ be the positive-definite matrix in Proposition \ref{prop:stableA}. The following statements are equivalent:
    \begin{enumerate}[label=(\alph*)]
        \item $\H$ is an energy-preserving tensor,
        \item $C_H = \bfz^\top \widetilde{\Q} \H : \bfz\bfz^\top + \left(\H : \bfz\bfz^\top\right)^\top\widetilde{\Q} \bfz = 0, \quad \forall \bfz \in \R^r$.
        \item Letting $(\H)_k$ be the $k$-th frontal slice of $\H$, i.e., $(\H)_k \coloneqq (\H)_{:,:,k}$, there exists a tensor $\widetilde{\S} \in \R^{r\times r \times r}$ with skew-symmetric slice $(\widetilde{\S})_k$ such that $(\H)_k = (\widetilde{\S})_k \widetilde{\Q}$.
    \end{enumerate}
\end{prop}

While these two propositions identify a pathway to strongly enforce global asymptotic stability in a quadratic dynamical system, parameterizing the tensors $\A$ and $\H$ in terms of skew-symmetric and positive (semi-)definite matrices can be inconvenient from an optimization standpoint.
In fact, skew-symmetric and positive (semi-)definite matrices do not form vector spaces.
To circumvent this problem, we move forward as suggested in \cite{goyal_guaranteed_2025} and parameterize any skew-symmetric matrix $\widetilde{\bm{K}}$ and positive semi-definite matrix $\widetilde{\bm{R}}$ as
\begin{equation}
\label{eq:matrix_param}
    \widetilde{\bm{K}} = \bm{K} - \bm{K}^\top,\quad \widetilde{\MR} = \MR\MR^\top,
\end{equation}
respectively, where $\bm{K}$ and $\MR$ are unconstrained square matrices that lend themselves to straightforward vector calculus and optimization.
The authors in \cite{goyal_guaranteed_2025} proposed an analogous parameterization for positive-definite matrices (i.e., $\widetilde{\Q} = \Q\Q^\top$, where $\Q$ is unconstrained), but they observed that gradient-descent-like algorithms had a tendency of converging to rank-deficient matrices $\Q$, ultimately leading to positive \emph{semi}-definite matrices $\widetilde{\Q}$.
In other words, their parameterization often led to decompositions that do not meet the requirements of Proposition \ref{prop:stableA}.
To solve this problem, we propose the following parameterization
\begin{equation}
\label{eq:pd_param}
    \widetilde{\Q} = \Q^{-1}\Q^{-\top},
\end{equation}
for an unconstrained $\Q$. 
This equality holds so long as $\Q$ is full rank, so this parameterization guarantees positive definiteness by construction.
Finally, using the results in Propositions \ref{prop:stableA} and \ref{prop:stableH}, and the proposed parameterizations, the latent-space dynamics may be expressed by equation~\eqref{eq:quadsys} with the additional requirement
\begin{equation}
    \A = \left(\K - \K^\top - \MR\MR^\top\right)\Q^{-1}\Q^{-\top},\quad \H_{ijk} = \Bigl(\S_{ilk} - \S_{kli}\Bigr)\Q^{-1}_{la}\Q^{-1}_{ja}.
\end{equation}
This guarantees global asymptotic stability of the reduced-order model.

\subsection{GasNiTROM}
Building upon the recently-introduced NiTROM framework~\cite{padovan_data-driven_2024} and the results in the previous section, we now develop a formulation to compute non-intrusive ROMs with dynamically-optimal encoder/decoder pairs and globally asymptotically stable latent-space dynamics.
We refer to this method as GasNiTROM (Globally-asymptotically-stable Non-intrusive Trajectory-based optimization of Reduced-Order Models).
% The objective of the NiTROM formulation is to compute dynamically-optimal reduced-order models of the form~\eqref{eq:nitromsys} by simultaneously optimizing encoder/decoder pairs and latent-space dynamics against a forecasting loss function.
Before stating the optimization problem, it is useful to expose some properties of the ROM $G(\bfphi,\bfpsi,\bff_r)$ in equation \eqref{eq:nitromsys}.
Specifically, as discussed in \cite{padovan_data-driven_2024}, it may be verified that the ROM $G$ is invariant under a rotation and scaling of the matrix $\bfphi$ (i.e., $G(\bfphi\W, \bfpsi, \bff_r) = G(\bfphi, \bfpsi, \bff_r)$ for any invertible $r\times r$ matrix $\W$).
This suggests that $G$ is a function of the $r$-dimensional subspace $V = \text{span}(\bfphi)$ rather than~$\bfphi$ itself.
Therefore, in the mathematical formulation of the problem, instead of seeking optimal frames $\bfphi$ of size $n\times r$, we will leverage this symmetry by parameterizing the ROM $G$ over the Grassmann manifold $\G_{n,r}$ of $r$-dimensional subspaces of $\R^n$.
% In the mathematical formulation of the problem, we will leverage the fact that $r$-dimensional subspaces of $\R^n$ live naturally on the Grassmann manifold $\G_{n,r}$ \cite{absil2008optimization}.
While an analogous symmetry does not hold for~$\bfpsi$, a necessary (but not sufficient) condition for the product~$\bfpsi^\top\bfphi$ to be invertible (see equation \eqref{eq:nitromsys}) is for~$\bfpsi$ to have full column rank.
To satisfy this condition, we therefore treat $\bfpsi$ as an element of the Stiefel manifold $S_{n,r}$, which is the space of orthonormal (and hence full-column-rank) frames of size $n\times r$.
We are now ready to state the GasNiTROM optimization problem.
Given ground-truth measurements $\bfy(t_i)$ sampled at times $t_i$ from the FOM in equation \eqref{eq:goveq}, we solve 
\begin{equation}
\label{eq:optproblem}
    \begin{aligned}
        \min_{\left(V,\bfpsi,\K,\MR,\Q,\S,\B\right) \in \M}\,\,\, J &=  \sum_{i=0}^{N-1}\lVert \bfy(t_i) - \hat{\bfy}(t_i)\rVert^2 \\
        \text{subject to:} \quad \frac{\rmd \bfz}{\rmd t} &= \bff_r(\bfz,\bfu)\coloneqq \A\bfz + \H:\bfz\bfz^\top + \B\bfu,\quad \bfz(t_0) = \bfpsi^\top \bfx(t_0) \\
            \A &= \left(\K - \K^\top - \MR\MR^\top\right)\Q^{-1}\Q^{-\top}\\
            \H_{ijk} &= \Bigl(\S_{ilk} - \S_{kli}\Bigr)\Q^{-1}_{la}\Q^{-1}_{ja}\\
            \hat{\bfy} &= \bfh\left(\bfphi\left(\bfpsi^\top \bfphi\right)^{-1}\bfz\right) \\
            V &= \text{Range}\left(\bfphi\right),
    \end{aligned}
\end{equation}
where $\M = \G_{n,r} \times S_{n,r} \times \R^{r \times r} \times \R^{r \times r} \times \R^{r \times r} \times \R^{r \times r \times r} \times \R^{r \times m}$ is the product manifold defining the optimization space.
In words, we simultaneously seek optimal encoders $\bfpsi^\top$, decoders $\bfphi\left(\bfpsi^\top\bfphi\right)^{-1}$, and globally asymptotically-stable latent-space dynamics $\bff_r$ by minimizing a forecasting loss function.

Since the program \eqref{eq:optproblem} is formulated over a nonlinear space, solving the problem using gradient-based methods requires notions from matrix manifold optimization \cite{absil2008optimization}.
Here, we provide enough details to derive a closed-form expression for the gradient of the cost function with respect to the optimization parameters, but we refer to section 2.3 in \cite{padovan_data-driven_2024} for a more thorough explanation.
The strategy is to express the gradient of $J$ in terms of matrix-/tensor-valued objects living in linear spaces (often referred to as ambient spaces) that lend themselves to straightforward vector calculus.
We begin by viewing $\M$ as a submanifold of an ambient-space manifold $\overline{\M}$ equipped with a Riemannian metric, which we shall define momentarily. 
Viewing the Stiefel manifold as an embedded submanifold of the vector space~$\R^{n \times r}$, and the Grassmann manifold as a quotient manifold of the non-orthogonal Stiefel manifold~$\R_{*}^{n \times r}$~\cite{absil2008optimization}, the ambient-space manifold may be defined component-wise as follows
\begin{equation}
    \overline{\M} = \R_{*}^{n \times r} \times \R^{n \times r} \times \R^{r \times r} \times \R^{r \times r} \times \R^{r \times r} \times \R^{r \times r \times r} \times \R^{r \times m}.
\end{equation}
% The manifolds $\R^{r \times r}$, $\R^{r \times r \times r}$, and $\R^{r \times m}$ are linear vector spaces that do not require any special treatment.
Next, it is now necessary to endow the ambient-space manifold $\overline\M$ with Riemannian metrics, which can then be inherited by $\M$.
Recalling that a Riemannian metric is a smooth family of inner products defined on the tangent spaces of the manifold $\M$,
\begin{equation}
    g_p^\M:\mathcal{T}_p\M\times \mathcal{T}_p\M \to \mathbb{R},\quad p\in\mathcal{M},
\end{equation}
we proceed as in section 2.3 in \cite{padovan_data-driven_2024} and define the following metrics for $\G_{n,r}$ and $St_{n,r}$, respectively,
\begin{equation}
    \begin{aligned}
        g_V^{\G_{n,r}}\left(\xi,\eta\right) &= g_{\bfphi}^{\R_*^{n\times r}}\left(\overline{\xi}_{\bfphi}, \overline{\eta}_{\bfphi}\right)\coloneqq \text{Tr}\left(\left(\bfphi^\top\bfphi\right)^{-1}\overline{\xi}_{\bfphi}^\top \overline{\eta}_{\bfphi}\right),\quad \xi,\eta\in\mathcal{T}_{V}\G_{n,r},\quad \overline{\xi}_{\bfphi},\,\overline{\eta}_{\bfphi}\in\mathcal{T}_{\bfphi} \R_{*}^{n\times r}\\ 
        g_{\bfpsi}^{St_{n,r}}(\xi,\eta) &\coloneqq \text{Tr}\left(\xi^\top\eta\right),\quad \xi, \eta \in \mathcal{T}_{\bfpsi}St_{n,r}.
    \end{aligned}
\end{equation}
Here, $\text{Tr}(\cdot)$ denotes the trace operator for matrices. 
The metrics for the Euclidean spaces $\mathbb{R}^{r\times r}$, $\mathbb{R}^{r\times r\times r}$, etc. are taken to be the usual tensor dot product.
Now, following \cite{padovan_data-driven_2024} and \cite{absil2008optimization}, the gradient $\nabla J$ may be expressed as
\begin{equation}
    \nabla J = \left(\nabla_{\bfphi} \overline{J}, \PP_{\bfpsi} \nabla_{\opsi} \overline{J}, \nabla_{\K} \overline{J}, \nabla_{\MR} \overline{J}, \nabla_{\Q} \overline{J}, \nabla_{\S} \overline{J}, \nabla_{\B} \overline{J}\right),\quad \left(\bfphi, \overline{\bfpsi}, \K, \MR, \Q, \S, \B\right) \in \overline\M,
\end{equation}
where $\overline{J}$ is the ambient-space cost function defined on the ambient-space manifold $\overline{\M}$, and $\mathbb{P}_{\bfpsi}$ denotes the orthogonal projection onto the tangent space of $S_{n,r}$ at $\bfpsi$ \cite{absil2008optimization}.
This equation emphasizes that the gradient of the cost function~$J$ defined over the abstract matrix manifold $\mathcal{M}$ can be conveniently computed from the gradient of $\overline{J}$.
Intuitively speaking, the function $\overline{J}$ is the analogue of $J$, except that it is viewed as a function of matrix-/tensors-valued objects on $\overline\M$.
Remarkably, the gradient of the ambient-space function $\overline{J}$ can be computed in closed form, thereby avoiding the need for potentially expensive automatic differentiation.
The closed-form expression of the gradient is provided in the proposition below.

\begin{prop} \label{prop:gradients}
    Let problem \eqref{eq:optproblem} be written as an equivalent
    unconstrained optimization problem with ambient-space Lagrangian $\overline{L}: \overline{\M} \to \R$ defined as 
    \begin{equation} \label{eq:lagrangian}
        \begin{aligned}
            &\overline{L}\left(\bfphi, \opsi, \K, \MR, \Q, \S, \B\right) =\sum_{i=0}^{N-1} \biggl\{ \overline{J}_i
            + \int_{t_0}^{t_i} \lam_i^\top \left( \frac{\rmd \bfz}{\rmd t} - \widetilde{\bff}_r\left(\bfz, \bfu\right) \right) \rmd t + \lam_i(t_0)^\top \left(\bfz(t_0) - \opsi^\top \bfx(t_0)\right)\biggl\},
        \end{aligned}
    \end{equation}
    where $\overline{J}_i = \left\lVert \bfe(t_i)\right\rVert^2$ with $\bfe(t_i) \coloneqq \bfy(t_i) - \bfh\left(\bfphi\left(\opsi^\top \bfphi\right)^{-1} \bfz(t_i)\right)$, and $\lam_i(t) \in \R^r$ is the $i$-th Lagrange multiplier with $t \in [t_0, t_i]$. Defining $\C_{j,k} \coloneqq \partial \bfh_j/\partial \bfx_k$, the gradient of $\overline{L}$ is given component-wise as
    
    \begin{align}
        \nabla_{\bfphi} \overline{L} &= \left\{ -2 \sum_{i=0}^{N-1} \left( \bm{I} - \opsi \left(\bfphi^\top \opsi\right)^{-1} \bfphi^\top\right) \C(t_i)^\top \bfe(t_i) \bfz(t_i)^\top \left(\bfphi^\top \opsi\right)^{-1}\right\} \left(\bfphi^\top \bfphi\right), \\
        \nabla_{\opsi} \overline{L} &= \sum_{i=0}^{N-1} \left(2 \bfphi \left(\opsi^\top \bfphi\right)^{-1} \bfz(t_i)\bfe(t_i)^\top \C(t_i)\bfphi \left(\opsi^\top \bfphi\right)^{-1} - \bfx(t_0) \lam_i(t_0)^\top\right), \\
        \nabla_{\K} \overline{L} &= \left(\nabla_{\A} \overline{L}\right) \Q^{-1} \Q^{-\top} - \Q^{-1} \Q^{-\top} \left(\nabla_{\A} \overline{L}\right)^\top, \label{eq:gradJ} \\
        \nabla_{\MR} \overline{L} &= -\left( \left(\nabla_{\A} \overline{L}\right) \Q^{-1} \Q^{-\top} + \Q^{-1} \Q^{-\top} \left(\nabla_{\A} \overline{L}\right)^\top \right) \MR, \\
        \nabla_{\Q} \overline{L} &= -\Q^{-\top} \left(\nabla_{\widetilde{\Q}} \overline{L} + \left(\nabla_{\widetilde{\Q}} \overline{L}\right)^\top\right) \Q^{-1} \Q^{\top}, \\
        \nabla_{\widetilde{\Q}} \overline{L} &= \left(\K - \K^\top - \MR \MR^\top\right) \nabla_{\A}\overline{L} + \sum_{k=1}^r \left(\left((\S)_k^\top - (\S)_k\right) \left(\nabla_{\H} \overline{L}\right)_k\right), \label{eq:gradQtilde} \\
        \nabla_{\left(\S\right)_k} \overline{L} &= \left(\nabla_{\left(\H\right)_k} \overline{L}\right) \Q^{-1} \Q^{-\top} - \Q^{-1} \Q^{-\top} \left(\nabla_{\left(\H\right)_k} \overline{L}\right)^\top, \quad \forall k \in \left\{1, 2, \ldots, r\right\}, \\
        \nabla_{\B} \overline{L} &= -\sum_{i=0}^{N-1} \int_{t_0}^{t_i} \lam_i \bfu^\top\,\rmd t.
    \end{align}
    Here,
    \begin{align}
        \nabla_{\A} \overline{L} = -\sum_{i=0}^{N-1} \int_{t_0}^{t_i} \lam_i \bfz^\top\,\rmd t,\quad \nabla_{\H} \overline{L} = -\sum_{i=0}^{N-1} \int_{t_0}^{t_i} \lam_i \otimes \bfz \otimes \bfz\,\rmd t,
    \end{align}
    and the Lagrange multiplier $\lam_i\left(t\right)$ satisfies the reduced-order adjoint equation
    \begin{equation}
        -\frac{\rmd \lam_i}{\rmd t} = \left[ \partial_{\bfz} \widetilde{\bff}_r \left(\bfz\right)\right]^\top \lam_i, \quad \lam_i(t_i) = 2\left(\bfphi^\top \opsi\right)^{-1} \bfphi^\top \C(t_i)^\top \bfe(t_i), \quad t \in [t_0, t_i].
    \end{equation}
\end{prop}
\begin{proof}
    See Appendix \ref{ap:proof}.
\end{proof}
% The authors remark that the above gradients of the Lagrangian with respect to the parametrized tensors depend only on having the gradients with respect to the overall tensors $\nabla_{\A} \overline{L}$ and $\nabla_{\H} \overline{L}$, and as such it can be generalized to any gradient-based optimization method where gradients are known. 

Once the ambient-space gradient is known, gradient-based optimization on matrix manifolds can be performed using open source, readily available software like \texttt{pyManOpt} \cite{pymanopt}.
For more details regarding the theory behind matrix-manifold optimization algorithms and their implementation via retractions and vector transport, we refer the reader to~\cite{absil2008optimization}.
% During gradient-based optimization on nonlinear manifolds, the concept of a \textit{retraction} is necessary to guarantee all points computed during optimization remain on the manifold. The retraction is defined as a map $R_p : \T_p \M \to \M$ that satisfies $R_p(0) = p$ and $DR_p(0) = I_{\T_p \M}$, where $I_{\T_p \M}$ is the identity map on the tangent space $\T_p \M$ \cite{absil2008optimization}. Given a point $p \in \M$  and the gradient $\xi \in \T_p \M$ of a function $f$ defined on $\M$, the next iterate in the direction of the gradient is given by $R_p(p - \beta \xi)$, where $\beta$ is some learning rate. Valid retractions for both the Grassmann and Stiefel manifolds are given by the QR decomposition, whereas for linear manifolds the retraction is simply the identity map.
% The authors of \cite{padovan_data-driven_2024} provide an estimate of the computational cost of the NiTROM algorithm in section 2.4. Due to the similarity with NiTROM and our formulation, the process for computing the ambient-space gradient efficiently is nearly the same as algorithm 2.1 in \cite{padovan_data-driven_2024}, with minor modifications for the parametrized tensors. The computational cost also remains the same, as the cost of the matrix-matrix and matrix-vector products of the parameterized tensors is negligible compared to the cost of matrix-vector products involving $\bfphi$ and $\bfpsi$ in the case of very high dimensional full-order systems, or the ROM time stepper in all other cases.

\subsection{Connection with existing methods}
Our formulation is closely aligned with the recently-introduced NiTROM formulation \cite{padovan_data-driven_2024}, as well as with Operator Inference \cite{peherstorfer_data-driven_2016} and its globally-stable version discussed in \cite{goyal_guaranteed_2025}. Here, we provide a brief review of these methods and identify similarities and differences.

\subsubsection{NiTROM}
Within the NiTROM framework \cite{padovan_data-driven_2024}, one solves an optimization problem analogous to \eqref{eq:optproblem}, except that no constraints are imposed on the form of the latent-space tensors $\A$ and $\H$.
% In particular, the optimization problem is exactly the same as \eqref{eq:optproblem} but constrains the dynamics to be of the form \eqref{eq:quadsys}, with dynamics governed by unconstrained tensors $\A$ and $\H$.  
This implies that, in principle, an optimal solution of the NiTROM problem can exhibit infinite-/finite-time blow up.
% Consequently, the ROMs obtained from NiTROM may be stable and accurate when considering the training data set, but could lead to exponential or finite-time blow up of solutions when testing against an unseen set of data. 
The authors in \cite{padovan_data-driven_2024} showed that linear stability of the dynamics can be promoted by penalizing the Frobenius norm of the matrix exponential $e^{\A t_f}$ evaluated at some very large $t_f \gg 1$, but no stability-promoting penalties were proposed for the higher-order polynomial terms in the latent-space dynamics.
By contrast, the globally-stable NiTROM procedure proposed in this manuscript ensures global asymptotic stability by construction, and does not require the use of penalties that potentially require careful tuning.

% This can be remedied in practice by using penalty methods, where a random forcing or initial condition is used, and the energy of the resulting trajectory is added to the cost $J$. These methods can bias the optimization towards stable solutions, but it is not guaranteed to converge to one. Our approach ensures global stability in the learned operators by construction, and does not require any additional methods to ensure the stability of solutions holds past the training data set.

\subsubsection{Operator Inference}

Operator Inference (OpInf) \cite{peherstorfer_data-driven_2016} is a non-intrusive model reduction framework that seeks a reduced-order model by orthogonally projecting the data onto a low-dimensional subspace and then fitting polynomial latent-space dynamics. 
The subspace used for the projection is typically chosen as the span of the leading proper orthogonal decomposition (POD) modes generated from the full-order data set. 
In particular, given a full-order trajectory $\bfx(t_i)$ sampled from \eqref{eq:goveq} at times $t_i$, the time-derivative $\rmd \bfx(t_i)/\rmd t$, the input forcing $\bfu(t_i)$, and an $r$-dimensional subspace spanned by $\bfphi \in \R^{n \times r}$, Operator Inference solves
\begin{equation} \label{eq:opinf}
    \min_{(\A, \H, \B) \in \M_{\text{OpInf}}} \quad J_{\text{OpInf}} = \sum_{i=0}^{N-1} \left\lVert \frac{\rmd \bfz(t_i)}{\rmd t} - \A \bfz(t_i) - \H: \bfz(t_i) \bfz(t_i)^\top - \B \bfu(t_i) \right\rVert^2_2,
\end{equation}
where $\bfz(t_i) = \bfphi^\top \bfx(t_i)$ and $\M_{\text{OpInf}} \in \R^{r \times r} \times \R^{r \times r \times r} \times \R^{r \times m}$. 
Equation \eqref{eq:opinf} can be equivalently cast as a linear least-squares problem, whose solution is given via the pseudoinverse rather than via iterative gradient-based algorithms. 
While Operator Inference provides a fast and convenient non-intrusive model-reduction framework, it is well-known that latent-space identification via orthogonal projection may lead to dynamically inaccurate models, especially in systems (e.g., fluid flows) exhibiting large-amplitude transient growth and high sensitivity to external disturbances.
Additionally, users often observe that this formulation has a tendency of producing unstable ROMs, which has led to the development of Tikhonov-regularized formulations that seek to promote stability by penalizing the Frobenius norm of the tensors $\A$ and $\H$ \cite{mcquarrie_data-driven_2021, sawant_physics-informed_2023}.
The formulation of \cite{goyal_guaranteed_2025} discussed below solves the issue of stability at its root by strongly enforcing stability of the latent-space dynamics.

\subsubsection{Guaranteed-stable Operator Inference}
Building on the Operator Inference framework, \cite{goyal_guaranteed_2025} proposes a guaranteed-stable formulation by parameterizing the latent-space tensors as discussed in section \ref{sec:formulation}.
The optimization problem remains of the same form as \eqref{eq:opinf} with the additional Lyapunov-stable parameterization of the reduced-order tensors.
Following the notation of section \ref{sec:formulation}, globally-asymptotically-stable Operator Inference (GasOpInf) solves
\begin{equation} \label{eq:gasopinf}
    \begin{aligned}
        \min_{(\K, \MR, \Q, \S, \B) \in \M_{\text{GasOpInf}}} \quad J_{\text{GasOpInf}} &= \sum_{i=0}^{N-1} \left\lVert \frac{\rmd \bfz(t_i)}{\rmd t} - \A \bfz(t_i) - \H: \bfz(t_i) \bfz(t_i)^\top - \B \bfu(t_i) \right\rVert^2_2 \\
        \text{subject to:} \quad \A &= \left( \K - \K^\top - \MR \MR^\top \right) \Q^{-1} \Q^{-\top} \\
        \H_{ijk} &= \left(\S_{ilk} - \S_{kli} \right) \Q_{la}^{-1} \Q_{ja}^{-1},
    \end{aligned}
\end{equation}
where $\bfz(t_i) = \bfphi^\top \bfx(t_i)$ and $\M_{\text{GasOpInf}} \in \R^{r \times r} \times \R^{r \times r} \times \R^{r \times r} \times \R^{r \times r \times r} \times \R^{r \times m}$.
Contrary to classical Operator Inference, \eqref{eq:gasopinf} is nonlinear in the parameters, therefore requiring iterative gradient-based algorithms.
While this formulation was introduced in \cite{goyal_guaranteed_2025}, we emphasize two minor contributions provided in this paper that augment the original GasOpInf.
First, as discussed previously, the authors in \cite{goyal_guaranteed_2025} parameterized the positive-definite matrix $\widetilde{\Q}$ in Proposition \ref{prop:stableA} as $\widetilde{\Q} = \Q\Q^\top$, which often leads to degenerate positive semi-definite matrices $\widetilde{\Q}$ when $\Q$ loses rank.
Instead, here we write $\widetilde{\Q} = \Q^{-1}\Q^{-\top}$, as shown in the first constraint in equation \eqref{eq:gasopinf}.
Since this expression is well-defined only for full-rank $\Q$, positive-definiteness of $\widetilde{\Q}$ is enforced by construction without any degeneracies.
Second, while the authors in \cite{goyal_guaranteed_2025} solve the optimization problem using automatic differentiation, we provide a convenient closed-form expression for the gradients of $J_{\text{GasOpInf}}$, which may be evaluated at lower computational cost, thereby accelerating the training phase. 
These gradients are provided in Appendix \ref{ap:opinf_grads}.
The GasOpInf formulation solves the previously-discussed stability issues of the classical Operator Inference formulation, but much like Operator Inference, it relies on orthogonal projections for dimensionality reduction.
% Additionally, Tikhonov regularization on the tensors $\A$ and $\H$ can be used in this formulation, but unlike classical Operator Inference, the regularization here works to avoid overfitting of the ROM.
As previously discussed, this can lead to inaccurate reduced-order models, especially when the full-order model exhibits large-amplitude transients and high sensitivity to disturbances.

% It is apparent that our model reduction framework builds upon the concept of NiTROM, embedding global stability of the resulting reduced-order model \eqref{eq:nitromsys} by parametrization of the latent-space dynamical operators. Furthermore, our formulation is more robust to the normality of the full-order dynamics due to the simultaneous optimization of latent-space dynamics with oblique projection matrices, leading to a more accurate ROM than if the same parameterization approach was performed with the Operator Inference framework. 

\section{Application to a toy model} \label{sec:toymodel}
In this section, we apply GasNiTROM to a three-dimensional toy model, and we compare with the intrusive POD-Galerkin formulation, classical and guaranteed-stable Operator Inference, and classical NiTROM. 
The model is governed by the following equations:
\begin{equation}\label{eq:toymodel}
    \begin{aligned}
        \dot{x}_1 &= -x_1 + \nu x_1 x_3 + u, \\
        \dot{x}_2 &= -2 x_2 + \nu x_2 x_3 + u, \\
        \dot{x}_3 &= -5 x_3 + u, \\
        y &= x_1 + x_2 + x_3,
    \end{aligned}
\end{equation}
where $\dot{x}_1 = \rmd x_1/\rmd t$ and $\nu$ is a coupling parameter. 
If $\nu$ is small, then these dynamics are effectively linear and governed by a diagonal linear operator. 
By contrast, if $\nu$ is large, not only do the dynamics become more heavily nonlinear, but the dynamics of $x_1$  and $x_2$ depend strongly on a rapidly-vanishing $x_3$.
Systems exhibiting high-sensitivity to low-energy/rapidly-vanishing states are precisely those where ROMs obtained via orthogonal projection are more likely to give inaccurate and unstable predictions.
In these cases, not only is it advisable to seek ROMs with guaranteed-stable latent-space dynamics, but it is also necessary to construct appropriate encoder/decoder pairs capable of detecting the strong dependence of high-energy/slowly-decaying states (here, $x_1$ and $x_2$) on low-energy/rapidly-decaying states (here, $x_3$).
For demonstrative purposes, we consider a large $\nu = 20$, and we seek two-dimensional ROMs capable of predicting the time history of the output $y$ in response to scalar inputs $u(t) = \gamma H(t)$, where $H(t)$ is the Heaviside step function centered at $t=0$, and $\gamma \in (0, 1/4)$. 
In particular, we seek quadratic ROMs of the form
\begin{equation}
    \begin{aligned}
        \frac{\rmd \bfz}{\rmd t} &= \A \bfz + \H: \bfz\bfz^\top + \bfpsi^\top \bfu, \\
        \hat{y} &= \C \bfphi \left(\bfpsi^\top \bfphi\right)^{-1} \bfz,
    \end{aligned}
\end{equation}
where $\C = \begin{bmatrix} 1 & 1 &1 \end{bmatrix}$ is a row vector and $\bfu = \left(u,u,u\right)$.

We train the model in the same manner as \cite{padovan_data-driven_2024}, collecting trajectory data from $N_{\text{traj}}=4$ step responses generated with $\gamma = \{0.01, 0.1, 0.2, 0.248\}$ and initialized from the unforced base-flow. 
For each trajectory, we sample at $N=100$ equally-spaced times $t_i \in [0, 10]$. 
The cost function for NiTROM and GasNiTROM is
\begin{equation}\label{eq:cost_nit}
    J = \sum_{j=0}^{N_{\text{traj}}-1} \frac{1}{\alpha_j} \sum_{i=0}^{N-1} \left\lVert y^{(j)}(t_i) - \hat{y}^{(j)}(t_i)\right\rVert^2,
\end{equation}
with $\alpha_j = N_{\text{traj}} N \lVert \C \overline{\bfx}^{(j)} \rVert^2$, where $\overline{\bfx}^{(j)}$ is the steady state of trajectory $j$ in response to the step input magnitude $\gamma^{(j)}$. 
Both methods were initialized with $\bfpsi = \bfphi$ given by the leading two POD modes $\bfphi$ computed from the training step responses. The NiTROM reduced-order dynamics were initialized via Galerkin projection of the full-order dynamics onto the POD modes, while the GasNiTROM tensors were initialized using the parameterized tensors from GasOpInf.
The cost function for classical and guaranteed-stable Operator Inference is
\begin{equation}
    J = \sum_{j=0}^{N_{\text{traj}}-1} \frac{1}{\alpha_j} \sum_{i=0}^{N-1} \left\lVert \frac{\rmd \bfz^{(j)}(t_i)}{\rmd t} - \A \bfz^{(j)}(t_i) - \H:\bfz^{(j)}(t_i) \bfz^{(j)}(t_i)^\top - \bfphi^\top \bfu^{(j)}(t_i) \right\rVert^2 + \lambda \left\lVert \mat(\H) \right\rVert^2_F,
\end{equation}
where $\bfz=\bfphi^\top \bfx$, $\mat(H)$ denotes the matricization of the rank-3 tensor $\H$, and $\lambda$ is the regularization parameter. 
For classical Operator Inference, $\lambda \approx 10^{-7}$, and for guaranteed-stable Operator Inference, $\lambda \approx 10^{-8}$, chosen (approximately) to yield the model that minimizes the cost function in \eqref{eq:cost_nit}.
We emphasize that the regularization term in the GasOpInf optimization problem is used exclusively to avoid overfitting and not to promote stability, which is already enforced by construction.
The GasOpInf optimization was initialized by projecting the latent-space tensors from POD-Galerkin onto the stable subspace, and setting $\Q = \bm{I}$ where $\bm{I}$ is the identity matrix. 
It is important to note that the time-derivative $\rmd \bfz(t_i)/\rmd t$ is computed exactly, i.e. $\rmd \bfz(t_i)/\rmd t = \bfphi^\top \bff(\bfx(t_i))$, where $\bff$ denotes the right-hand side of the full-order dynamics. 
For all methods, the optimization was performed using the limited-memory Broyden–Fletcher–Goldfarb–Shanno (L-BFGS) algorithm \cite{goldfarb_family_1970} available in the \verb|PyTorch| library, and accelerated using GPUs, with the ambient-space gradient defined following Proposition \ref{prop:gradients}. 
Natively, \verb|PyTorch| does not support optimization on manifolds, so concepts such as \textit{metrics}, \textit{retractions}, and \textit{vector transport} \cite{absil2008optimization} were implemented manually using the \verb|PyTorch| backend.

The models were tested by generating 100 step-response trajectories with $\gamma$ sampled uniformly at random from the interval $(0, 1/4)$. The results are shown in Figure \ref{fig:toymodel_error}, where the average error over all trajectories $e(t)$ is defined as 
\begin{equation}\label{eq:toymodel_error}
    e(t) = \frac{1}{N_{\text{traj}}} \sum_{j=0}^{N_{\text{traj}}-1} \frac{1}{\alpha_j} \left \lVert y^{(j)}(t) - \hat{y}^{(j)}(t) \right \rVert^2,
\end{equation}
with $\alpha_j$ as in \eqref{eq:cost_nit}. Figure \ref{fig:toymodel_error} shows the testing error for two temporal windows: Figure \ref{fig:toymodel_err_10} is the error for all times within the window that was used for training ($t \in [0, 10]$), whereas Figure \ref{fig:toymodel_err_30} extends this window past the final training time ($t \in [0, 30]$) to determine the behavior of each model during temporal extrapolation. 
The GasNiTROM model obtains a lower overall error than the classical NiTROM model in Figure \ref{fig:toymodel_err_10}, and in Figure \ref{fig:toymodel_err_30} it can be seen that the NiTROM model becomes unstable, whereas the GasNiTROM model continues to attain a low average error. 
A more interesting result can be observed in Figure \ref{fig:toymodel_testing}, where we test the forecasting capabilities of the models on sinusoidal forcing functions of varying amplitude (recall that we have not trained on sinusoids). 
Two input amplitudes were considered for this study: a ``low'' amplitude of 0.45, and a ``high'' amplitude of 0.65. For the low amplitude case in Figure \ref{fig:toymodel_low_30}, all models remain stable for all time, with the GasNiTROM, GasOpInf, and NiTROM predictions attaining the smallest deviations from the full-order model. 
We emphasize that the NiTROM and GasNiTROM forecasts are the most accurate, since they account for the ``non-normal'' dynamics of the full-order model by identifying appropriate encoder/decoder pairs defining oblique projection operators.
By contrast, the GasOpInf framework, which relies on orthogonal projections onto the most energetic features of the full-order model, consistently exhibits higher error.
When the amplitude is increased in \ref{fig:toymodel_high_30}, POD-Galerkin, classical Operator Inference, and classical NiTROM all produce unstable trajectories, whereas GasOpInf and GasNiTROM remain stable, as expected.
% Of course, this is attributed to the parametrization of the latent-space dynamics, guaranteeing stability by construction.
While both the GasOpInf and GasNiTROM formulations are operating well outside their training regime and exhibit large errors, the GasNiTROM models is consistently more accurate, especially at predicting the peaks at times $t \in \{2, 8, 14,\ldots\}$.
% We emphasize that while both the GasOpInf and GasNiTROM formulations are operating well-outside their training regimes and exhibit substantial errors, the GasNiTROM framework remains visibly more accurate.
Once again, we attribute this feature to the fact that NiTROM-like architectures seek oblique-projection operators capable of detecting low-/high-energy coupling in the FOM dynamics, thus leading to more accurate ROMs.

\begin{figure}
    \centering
    \begin{subfigure}[b]{0.345\textwidth}
        \centering
        \includegraphics[width=\textwidth]{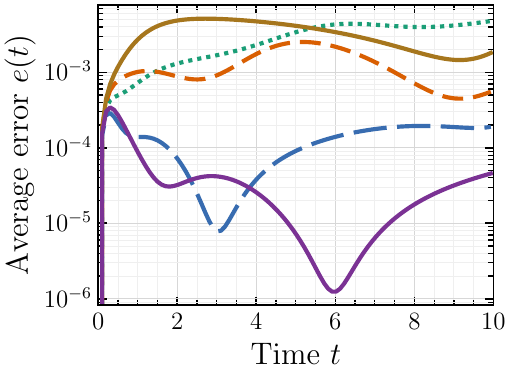}
        \caption{$t_{\text{final}}=10$}
        \label{fig:toymodel_err_10}
    \end{subfigure}\hfill
    \begin{subfigure}[b]{0.615\textwidth}
        \centering
        \includegraphics[width=\textwidth]{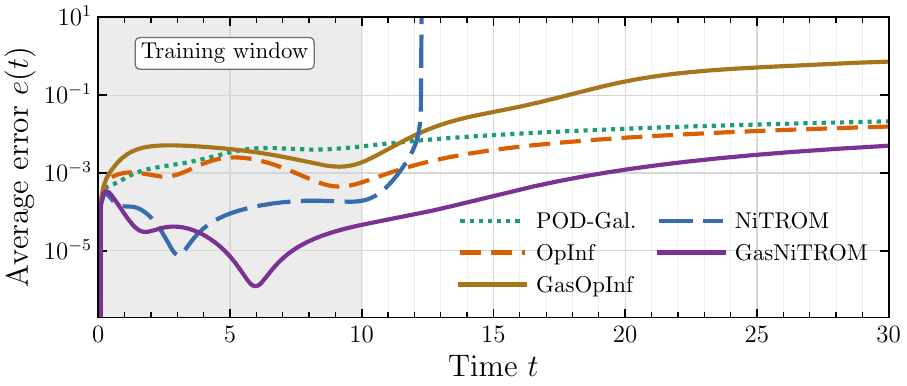}
        \caption{$t_{\text{final}}=30$}
        \label{fig:toymodel_err_30}
    \end{subfigure}
    \caption{Toy model average testing error \eqref{eq:toymodel_error} for final times of (a) 10 and (b) 30.}
    \label{fig:toymodel_error}
\end{figure}

\begin{figure}
    \centering
    \begin{subfigure}[b]{0.473\textwidth}
        \centering
        \includegraphics[width=\textwidth]{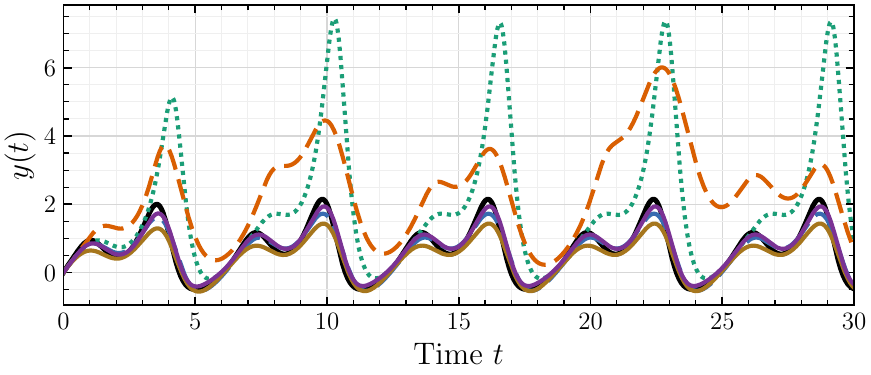}
        \caption{$u(t) = 0.45(\sin(t) + \cos(2t))$}
        \label{fig:toymodel_low_30}
    \end{subfigure}\hfill
    \begin{subfigure}[b]{0.487\textwidth}
        \centering
        \includegraphics[width=\textwidth]{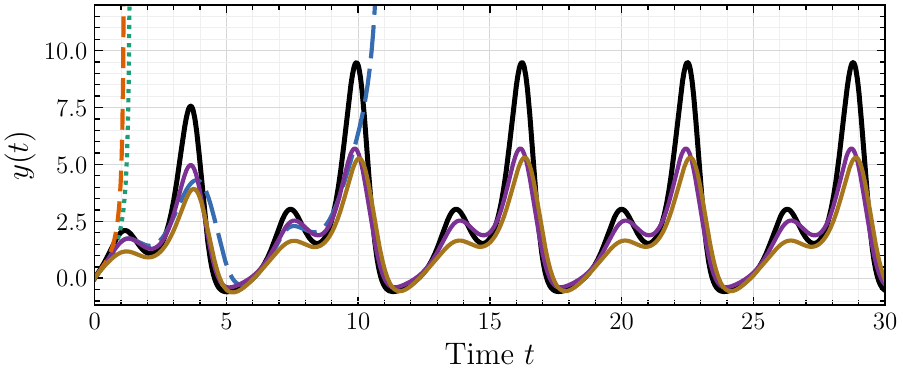}
        \caption{$u(t) = 0.65(\sin(t) + \cos(2t))$}
        \label{fig:toymodel_high_30}
    \end{subfigure}\hfill
    \caption{Time history of the toy model output $y$ in response to a sinusoidal input of amplitude (a) 0.45 and (b) 0.65. The black continuous line is the ground truth given by \eqref{eq:toymodel}. The rest of the legend is as in Figure \ref{fig:toymodel_err_30}.}
    \label{fig:toymodel_testing}
\end{figure}

\section{Application to the lid-driven cavity flow} \label{sec:cavity}
In this section, we apply our model reduction procedure to an incompressible fluid flow inside a lid-driven square cavity. The flow dynamics are governed by the incompressible Navier-Stokes equations
\begin{equation}
    \begin{aligned}
        \frac{\partial \bfv}{\partial t} + \bfv \cdot \nabla \bfv &= -\nabla p + Re^{-1} \nabla^2 \bfv, \\
        \nabla \cdot \bfv &= 0,
    \end{aligned}
\end{equation}
where $\bfv(\bfx,t) = (u(\bfx, t), v(\bfx,t))$ is the two-dimensional velocity vector, $p(\bfx,t)$ is the pressure, and $Re$ is the Reynolds number. We consider a two-dimensional spatial domain of length 1 in either direction with zero-velocity boundary conditions at all walls, except for $u=1$ at the top wall. The Reynolds number is fixed at $Re = 8300$, for which the flow admits a linearly stable steady state (shown in Figure \ref{fig:cavity_baseflow}), but exhibits large amplification of disturbances due to the non-normal nature of the underlying dynamics.
The high degree of transient growth is shown in Figure \ref{fig:cavity_pert_energy}, where we show the time history of the energy of several impulse responses. We discretize the governing equations using a second-order finite-volume scheme on a uniform fully staggered grid of $N_x \times N_y = 100 \times 100$. The temporal integration is performed using the second-order fractional step method introduced in \cite{chorin_numerical_1968}. Our solver was verified by reproducing results from \cite{ghia_high-re_1982}.

\begin{figure}
    \centering
    \begin{subfigure}[b]{0.4\textwidth}
        \centering
        \includegraphics[width=\textwidth]{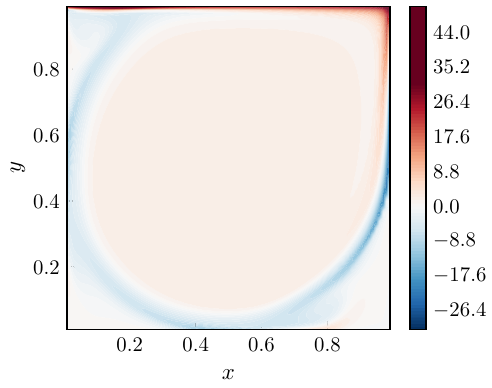}
        \caption{}
        \label{fig:cavity_baseflow}
    \end{subfigure}\hspace{1.5cm}
    \begin{subfigure}[b]{0.4\textwidth}
        \centering
        \includegraphics[width=\textwidth]{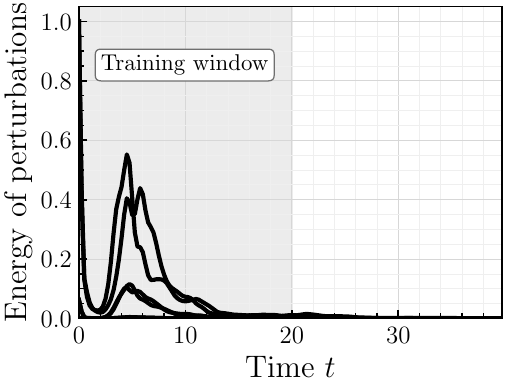}
        \caption{}
        \label{fig:cavity_pert_energy}
    \end{subfigure}
    \caption{(a) Vorticity field from the steady-state solution of the cavity flow at $Re = 8300$, and (b) the energy (i.e., squared two-norm) of the seven training trajectories.}
    \label{fig:cavity_base}
\end{figure}

For the sake of example, we are interested in computing reduced-order models capable of predicting the time-resolved response of the flow subject to spatially localized inputs that enter the $x$-momentum equation as
\begin{equation} \label{eq:cavityforcing}
    B_f(x,y)w(t) = \operatorname{exp}\left\{ -5000\left( (x-x_c)^2 + (y-y_c)^2\right) \right\} w(t),
\end{equation}
with $x_c = y_c = 0.95$. Upon spatial discretization and removal of the pressure via projection onto the space of divergence-free fields, the full-order dynamics are governed by
\begin{equation} \label{eq:cavityfom}
    \frac{\rmd \bfq}{\rmd t} = \A_f \bfq + \H_f : \bfq \bfq^\top + \B_f w,
\end{equation}
where $\bfq \in \R^N$ is the spatially discretized divergence-free velocity field (here, $N = 2 N_x N_y = \num{2e4}$), $\A_f$ governs the linear dynamics, $\H_f$ is a rank-3 tensor representation of the quadratic nonlinearity in the momentum equation, and $\B$ is the input matrix associated with \eqref{eq:cavityforcing} (we scale $\B_f$ to unit norm without loss of generality). Throughout this section, we take $\bfy = \bfq$ (i.e., we observe the whole state for training purposes).

\subsection{Training} \label{sec:cavity_training}
For training, we collect seven impulse trajectories.
These impulses are simulated by running \emph{unforced} simulations (i.e., $\bfu = 0$) initialized with initial conditions $\bfq(0) = \beta \B_f$ with
$\beta \in \{-1.0, -0.25, -0.05, 0.01, 0.05, 0.25, 1.0\}$.
The time history of the energy of the training trajectories is shown in Figure \ref{fig:cavity_pert_energy}. 
We save 160 snapshots from each trajectory at equally-spaced time intervals of $\Delta t = 0.25$, and then we perform POD. 
As in the previous section, we normalize the trajectories by their time-averaged energy. 
Using the first 50 POD modes, which contain $99.6\%$ of the variance of the full-order data, we compute a POD-Galerkin model, and we also compute Operator Inference and GasOpInf models by minimizing the objective functions \eqref{eq:opinf} and \eqref{eq:gasopinf}, respectively, penalizing the Frobenius norm of the rank-3 tensor $\H$ with the regularization parameter $\lambda = \num{3e-4}$ for OpInf and $\lambda = \num{9e-6}$ for GasOpInf.
We again remark that in the GasOpInf case, the regularization term exclusively prevents overfitting, with no additional stability-promoting benefits, since the ROM is stable by construction.
As discussed in \cite{padovan_data-driven_2024}, in order to guarantee that the optimal $\bfphi$ and $\bfpsi$ span divergence-free subspaces, we pre-project the FOM data onto the span of the first 200 POD modes, which contain $>99.99\%$ of the variance.
As an added benefit, this also slightly reduces the training cost, since it effectively reduces the size of the full-order model to $n = 200$.
In this case study, GasOpInf was initialized with the POD-Galerkin tensors projected onto the stable subspace as done previously, NiTROM is initialized with the POD-Galerkin model, and GasNiTROM with the GasOpInf model.
We then train by progressively extending the length of the time horizon.
That is, we first optimize a model over the time window $t \in [0, 1.25]$, then $t\in[0,2.5]$, all the way down to $t \in [0, 20]$. 
We explicitly do not extend the window until the final trajectory time of $t=40$ to explore the models' effectiveness in extrapolating on incomplete temporal data.
The training was conducted by using the L-BFGS algorithm as in \ref{sec:toymodel} until $t=7.5$ for both models due to the near second-order convergence towards a local minimum.
After $t=7.5$, the AdamW optimizer \cite{loshchilov_decoupled_2019}, natively available in \verb|PyTorch|, was used as the problem size became too large to query the gradient multiple times each iterative step.
A stability-promoting penalty term of the form used in Section 5.1 of \cite{padovan_data-driven_2024} was added to the NiTROM objective function after a first optimization pass due to the presence of unstable linear dynamics. The optimization was rerun on the entire $t=20$ temporal window until all eigenvalues of $\A$ were in the left-half of the complex plane.
Additionally, the training was conducted using the \textit{coordinate-descent} logic, whereby we successively fix the reduced-order tensors and optimize $\bfphi$ and $\bfpsi$, and vice versa.
It was observed that this coordinate-descent approach led to more robust training and avoided getting trapped in poor local minima.
This was done to provide more robustness to the training approach and avoid getting stuck in poor local minima. 

\subsection{Testing} \label{sec:cavity_testing}
In this section, we compare GasNiTROM against NiTROM, GasOpInf, Operator Inference, and POD-Galerkin. We test the models by generating 25 impulse responses with the impulse magnitude $\beta \in [-1,1]$ drawn uniformly at random. The training and testing errors for GasNiTROM, NiTROM, GasOpInf, Operator Inference, and the POD-Galerkin model (all with dimension $r=50$) are shown in Figures \ref{fig:cavity_error_train} and \ref{fig:cavity_error_test}, respectively. The error is defined as
\begin{equation} \label{eq:cavity_error}
    e(t) = \frac{1}{N_{\text{traj}}}\sum_{j=1}^{N_{\text{traj}}}\frac{1}{e_{j, \text{avg}}}\lVert \bfq^{(j)}(t) - \hat{\bfq}^{(j)}(t) \rVert^2,\quad e_{j,\text{avg}} = \frac{1}{N}\sum_{i=0}^{N-1}\lVert \bfq^{(j)}(t_i) \rVert^2
\end{equation}
where $\bfq$ is the ground truth and $\hat{\bfq}$ is the prediction given by the ROM. 
From the figures, up until the end of the training period of $t=20$, we see that the GasNiTROM and NiTROM models attain similar errors for all times, whereas the other models can be up to two orders of magnitude less accurate. Once the models extrapolate past the final training time, the error attained by the NiTROM model increases by about an order of magnitude, whereas the error from the GasNiTROM model remains low for all times.
Notably, unlike OpInf, POD-Galerkin, and GasOpInf, the NiTROM and GasNiTROM models do not exhibit a large error peak at time $t \approx 5$, when the transient energy of the flow peaks (see Figure \ref{fig:cavity_base}b), suggesting that our NiTROM-like architectures successfully detected the mechanisms responsible for energy growth inside the cavity.
% The most notable feature from these figures is the lack of a peak in error around $t=5$ (when the fluid exhibits the most transient growth) for NiTROM and GasNiTROM, which can be attributed to the oblique projection bases for these models.

\begin{figure}
    \centering
    \begin{subfigure}[b]{0.319\textwidth}
        \centering
        \includegraphics[width=\textwidth]{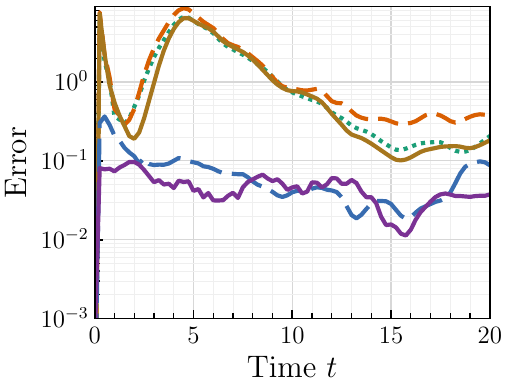}
        \caption{$t_{\text{final}}=20$}
        \label{fig:cavity_error_train_trained}
    \end{subfigure}\hfill
    \begin{subfigure}[b]{0.641\textwidth}
        \centering
        \includegraphics[width=\textwidth]{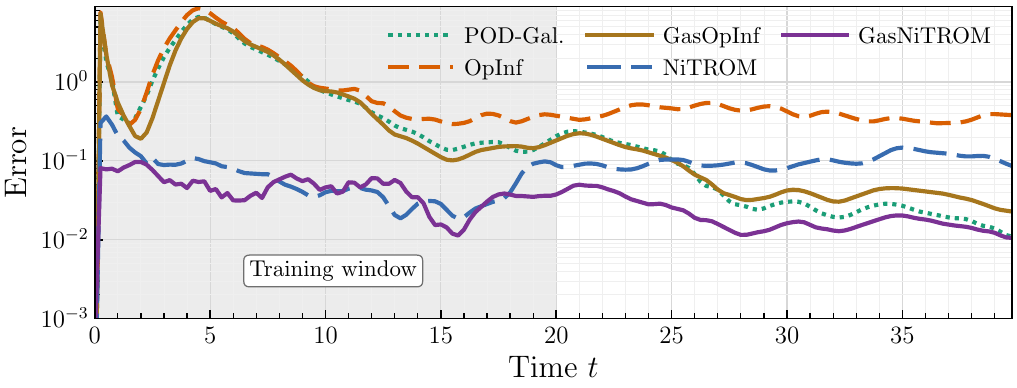}
        \caption{$t_{\text{final}}=40$}
        \label{fig:cavity_error_train_full}
    \end{subfigure}
    \caption{Cavity flow average training error \eqref{eq:cavity_error} for final times of (a) 20 and (b) 40.}
    \label{fig:cavity_error_train}
\end{figure}
\begin{figure}
    \centering
    \begin{subfigure}[b]{0.319\textwidth}
        \centering
        \includegraphics[width=\textwidth]{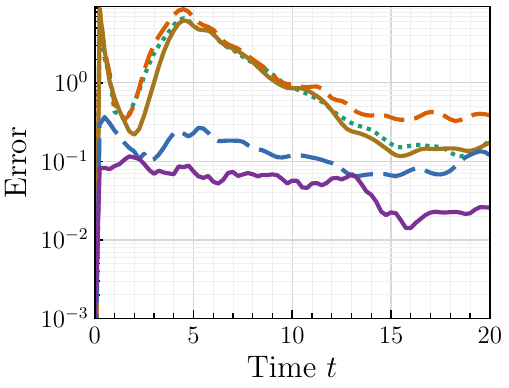}
        \caption{$t_{\text{final}}=20$}
        \label{fig:cavity_error_test_trained}
    \end{subfigure}\hfill
    \begin{subfigure}[b]{0.641\textwidth}
        \centering
        \includegraphics[width=\textwidth]{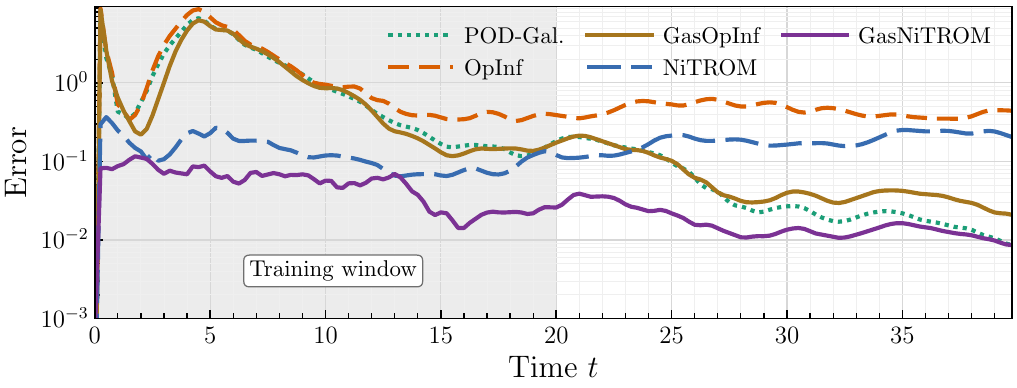}
        \caption{$t_{\text{final}}=40$}
        \label{fig:cavity_error_test_full}
    \end{subfigure}
    \caption{Cavity flow average testing error \eqref{eq:cavity_error} for final times of (a) 20 and (b) 40.}
    \label{fig:cavity_error_test}
\end{figure}

We also test the ability of our reduced-order model to predict the response of the system to sinusoidal inputs of the form $w(t)=a \sin(k \omega t)$, where $a \in \{0.1, 0.7\}$.
The results are shown in Figure \ref{fig:cavity_forcing_energy} where we see the response to harmonics of the angular frequency $\omega = 1$.
Across all cases, GasNiTROM performs as well, or better, than NiTROM. 
More specifically, for the low amplitude case of $a=0.1$ (see Figure \ref{fig:cavity_forcing_0p1_energy}), GasNiTROM correctly captures the steep energy growth at early times and saturates to a meaningful amplitude.
By contrast, the models based on orthogonal projections (OpInf, GasOpInf, and POD-Galerkin) miss the energy growth entirely, and settle to post-transient states with much lower amplitude than the ground truth (black line in the plots).
When the amplitude is increased to $a=0.7$, the flow behavior becomes more extreme and the models start operating even farther outside the training regime. 
In this regime, all models become less reliable, but on average GasNiTROM outperforms the rest, and is often capable of forecasting through the steep energy growth and onto a meaningful post-transient attractor.
We emphasize that the classical NiTROM model blows up in finite time for the $k=4$ case.
% For instance, Operator Inference attains a very large (but still stable) response due to the forcing, while the prediction from NiTROM blows-up in finite-time for $k=4$. For all three cases, GasNiTROM is generally capable of predicting the amplitude of the system response, with the prediction becoming unreliable for $k=2$ and $k=4$.
\begin{figure}
    \centering
    \begin{subfigure}[b]{0.483\textwidth}
        \centering
        \includegraphics[width=\textwidth]{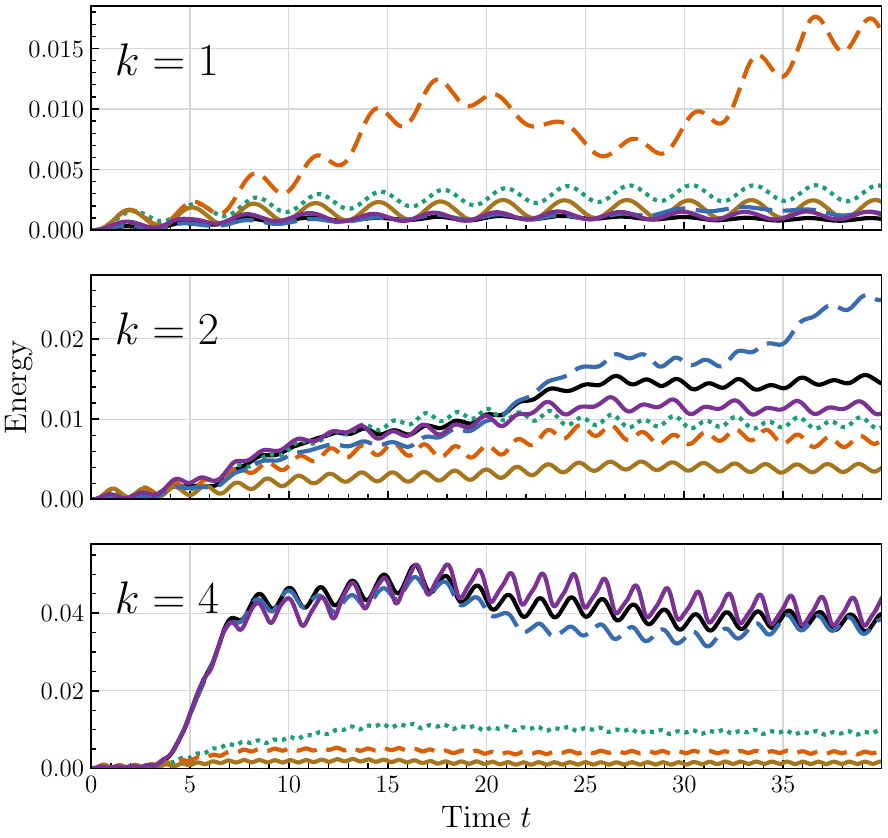}
        \caption{$a=0.1$}
        \label{fig:cavity_forcing_0p1_energy}
    \end{subfigure}\hfill
    \begin{subfigure}[b]{0.477\textwidth}
        \centering
        \includegraphics[width=\textwidth]{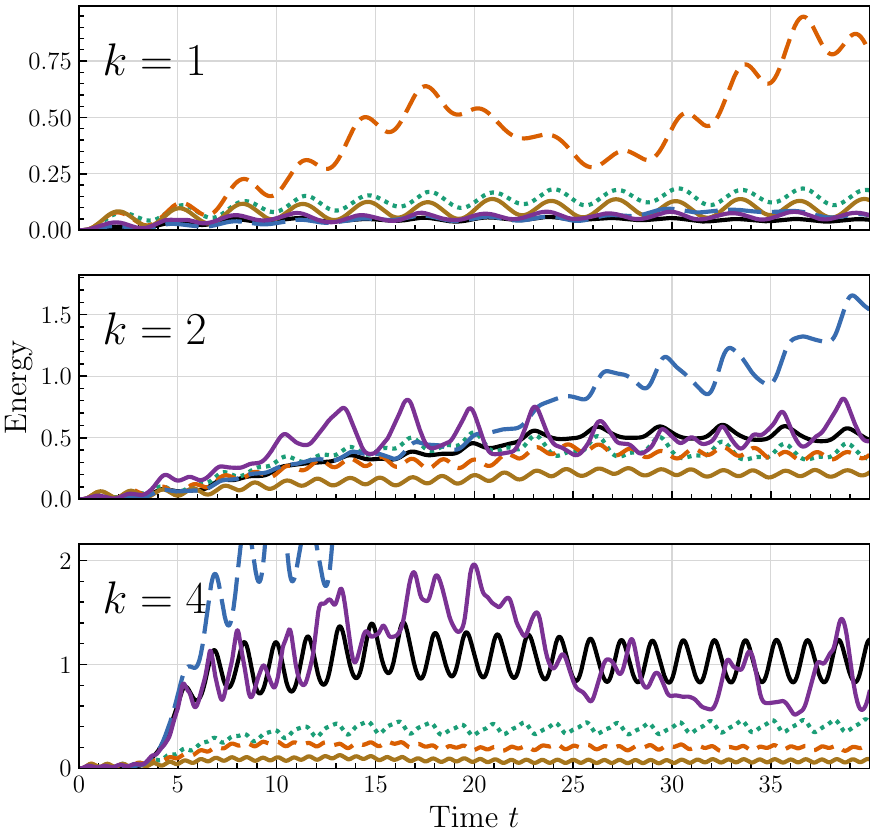}
        \caption{$a=0.7$}
        \label{fig:cavity_forcing_0p7_energy}
    \end{subfigure}
    \caption{Cavity flow time-response of the energy of the perturbations in response to sinusoidal inputs $w(t)=a\sin(kt)$ for an $a$ of (a) 0.1, (b) 0.7. The black line is the ground truth obtained by numerical simulation of the Navier-Stokes equations.}
    \label{fig:cavity_forcing_energy}
\end{figure}
Finally, we show vorticity snapshots at $t=30$ from the response of sinusoidal forcing with $k=4$ at $a=0.1$ and $a=0.7$.
The $a=0.1$ case is shown in Figure \ref{fig:cavity_forcing_snapshots_0p1}, where Operator Inference, GasOpInf, and POD-Galerkin underestimate the magnitude of the vorticity and predict the wrong phase of the vortical structures. 
Meanwhile, both NiTROM and GasNiTROM display very good agreement with the ground-truth data. 
In Figure \ref{fig:cavity_forcing_snapshots_0p7} the magnitude is increased to $a=0.7$ to test the models' effectiveness under more extreme full-order dynamics. 
For this case, NiTROM blows up, the GasOpInf, OpInf and POD-Galerkin models predict the wrong phase and amplitude of the vortical structures, while GasNiTROM predicts spurious vortical artifacts, albeit with what appears to be the correct amplitude.
For completeness, we have included a training scenario for the cavity flow where we compute optimal ROMs of dimension $r=30$ and train over the entire $t \in [0, 40]$ time window to determine the variation in model accuracy due to variation in the ROM dimension.
These results are available in Appendix \ref{ap:cavity_r30}.

\begin{figure}
    \centering
    \begin{subfigure}[b]{0.48\textwidth}
        \centering
        \includegraphics[width=\textwidth]{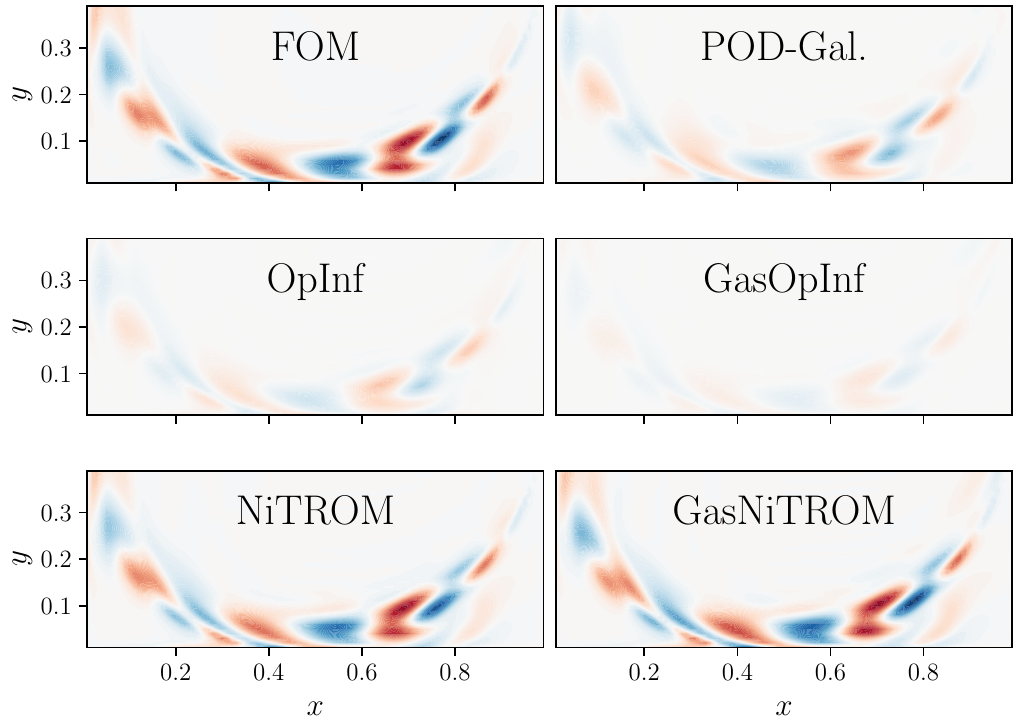}
        \caption{$a=0.1$, maximum vorticity 0.71}
        \label{fig:cavity_forcing_snapshots_0p1}
    \end{subfigure}\hfill
    \begin{subfigure}[b]{0.48\textwidth}
        \centering
        \includegraphics[width=\textwidth]{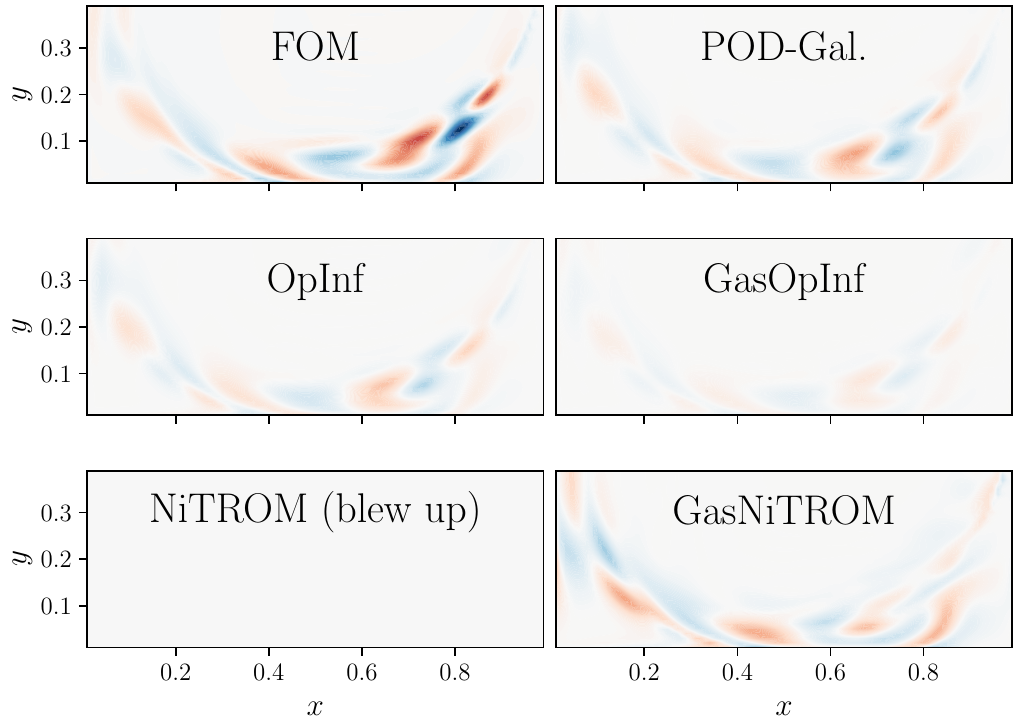}
        \caption{$a=0.7$, maximum vorticity 5.3}
        \label{fig:cavity_forcing_snapshots_0p7}
    \end{subfigure}
    \caption{Cavity flow vorticity field at time $t=30$ given sinusoidal forcing with (a) $a=0.1$, (b) $a=0.7$, $k=4$ in Figure \ref{fig:cavity_forcing_energy}. Red indicates positive vorticity with maximum value stated, blue indicates negative vorticity with minimum value of opposite sign of the maximum, and white is zero vorticity.}
    \label{fig:cavity_forcing_snapshots}
\end{figure}

% \begin{figure}
%     \centering
%     \includegraphics[width=0.7\textwidth]{figures/cavity_50_forcing_0p1_k4_snapshot_all}
%     \caption{Cavity flow vorticity field at time $t=30$ given sinusoidal forcing with $a=0.1$, $k=4$ in Figure \ref{fig:cavity_forcing_energy}. Red indicates positive vorticity with maximum value 0.71, blue indicates negative vorticity with minimum value -0.71, and white is zero vorticity.}
%     \label{fig:cavity_forcing_snapshots_0p1}
% \end{figure}
% \begin{figure}
%     \centering
%     \includegraphics[width=0.7\textwidth]{figures/cavity_50_forcing_0p7_k4_snapshot_all}
%     \caption{Cavity flow vorticity field at time $t=30$ given sinusoidal forcing with $a=0.7$, $k=4$ in Figure \ref{fig:cavity_forcing_energy}. Red indicates positive vorticity with maximum value 5.3, blue indicates negative vorticity with minimum value -5.3, and white is zero vorticity.}
%     \label{fig:cavity_forcing_snapshots_0p7}
% \end{figure}

\section{Conclusion}

In this paper, we have introduced a non-intrusive framework to learn stable and accurate reduced-order models of nonlinear high-dimensional systems exhibiting strong transient behavior.
% Examples of such systems include high--shear flows such as mixing layers, jets, and boundary layers, 
Several state-of-the-art non-intrusive model reduction methods fail to accurately capture the underlying mechanisms responsible for the presence of strong transients and often produce unstable models.
Recent developments have solved the issue of stability by enforcing known structure or symmetries in the model reduction architecture, or by parameterizing the latent-space dynamics so as to guarantee global asymptotic stability in the sense of Lyapunov.
Nonetheless, all these approaches achieve dimensionality reduction via orthogonal projections that often truncate the low-energy, dynamically-critical states that are responsible for the aforementioned high-energy transients observed in the full-order model.
In order to address this issue, we introduce a non-intrusive framework designed to simultaneously identify globally-asymptotically-stable latent-space dynamics and oblique projection operators capable of capturing the sensitivity mechanisms that drive the dynamics of the original system.
In particular, given training trajectories and a Lyapunov-based parameterization of quadratic polynomial reduced-order dynamics, we fit a model by optimizing over the product manifold of a Grassmann manifold, Stiefel manifold, and as many linear spaces as necessary to parameterize the dynamics under the stability constraint.
The resulting reduced-order model is guaranteed to be globally asymptotically stable (in a Lyapunov sense), and is also capable of forecasting through large-amplitude transients.
This framework is termed GasNiTROM -- ``Globally-asymptotically-stable Non-intrusive Trajectory-based optimization of Reduced-Order Models'' -- and it is demonstrated on two examples: a simple three-dimensional toy model system of ordinary differential equations and the two-dimensional incompressible lid-driven cavity flow at Reynolds number $Re=8300$.
In both cases, GasNiTROM outperforms state-of-the-art non-intrusive model reduction techniques (some of which blow up in finite time), and exhibits an improvement in forecasting accuracy of up to two orders of magnitude.
% Currently, GasNiTROM is limited to latent-space dynamics of quadratic polynomial form but, in the future, it would be interesting to explore the possibility of extending it to a general form of the dynamics, particularly those with non-polynomial terms.

\section*{Acknowledgments}
The authors wish to thank the  Texas Advanced Computing Center (TACC) and Nvidia teams at the 2025 TACC Open Hackathon for their assistance in the optimization of the NiTROM code base on GPUs.
This work was supported by the Center for Hypersonics and Entry Systems Studies in the Grainger College of Engineering at the University of Illinois Urbana-Champaign and by the National Science Foundation under grant 2139536, issued to the University of Illinois at Urbana-Champaign by the Texas Advanced Computing Center (TACC) under subaward UTAUS-SUB00000545 with Dr. Daniel Stanzione as the PI.
A. P. gratefully acknowledges startup support from the Newark School of Engineering at New Jersey Institute of Technology.

\bibliography{references}
\clearpage

\appendix
\section{Proof of Proposition \ref{prop:gradients}}
\label{ap:proof}
The authors of \cite{padovan_data-driven_2024} provide a detailed derivation of the gradients of the projection matrices $\bfphi$ and $\opsi$ and of the tensors $\A$, $\H$, and $\B$ in section 2.3, using calculus of variations. For the parametrized tensors, the proof relies further on calculus of variations. Recall that the tensor $\A$ is parametrized as in \eqref{eq:linstable}. Taking a small variation in $\A$, $\delta \A$, results in
\begin{equation}
    \delta \A = (\delta \widetilde{\K})\widetilde{\Q} - (\delta \widetilde{\MR})\widetilde{\Q} + (\widetilde{\K} - \widetilde{\MR})\delta \widetilde{\Q}.
\end{equation}
Denoting $G_\alpha = \nabla_\alpha \overline{L}$ for any $\alpha \in \{\A, \H, \K, \MR, \Q, \S, \B\}$, we take the inner product of the above with $G_{\A}$, as we know that the variation in $\overline{L}$, $\delta \overline{L}$ for a given $\delta \A$ is given by
\begin{equation}
    \delta \overline{L} = \inp*{G_{\A}}{\delta \A},
\end{equation}
which allows us to determine the gradients of parametrized tensors given the gradient of $\A$ at a point. In doing so, we arrive at
\begin{equation}
    \inp*{G_{\A}}{\delta \A} = \inp*{G_{\A}}{(\delta \widetilde{\K})\widetilde{\Q}} - \inp*{G_{\A}}{(\delta \widetilde{\MR})\widetilde{\Q}} + \inp*{G_{\A}}{(\widetilde{\K} - \widetilde{\MR})\delta \widetilde{\Q}}.
\end{equation}
Focusing on the first term on the right-hand side, we use the definition of the matrix inner product in $\R^n$ to write
\begin{equation}
    \inp*{G_{\A}}{(\delta \widetilde{\K}) \widetilde{\Q}} = \tr(G_{\A}^\top (\delta \widetilde{\K}) \widetilde{\Q}) = \tr(\widetilde{\Q} G_{\A}^\top \delta \widetilde{\K}) = \inp*{G_{\A} \widetilde{\Q}^\top}{\delta \widetilde{\K}}.
\end{equation}
Thus giving the gradient of $\overline{L}$ with respect to $\widetilde{\K}$ as
\begin{equation}
    G_{\widetilde{\K}} = G_{\A} \widetilde{\Q}^\top
\end{equation}
Similar methods can be used to obtain the gradients with respect to $\widetilde{\MR}$ and $\widetilde{\Q}$, respectively, as
\begin{equation}
    G_{\widetilde{\MR}} = -G_{\A} \widetilde{\Q}^\top, \quad G_{\widetilde{\Q}}^{(\A)} = (\widetilde{\K} - \widetilde{\MR})^\top G_{\A},
\end{equation}
where the additional notation in the gradient of $\overline{L}$ with respect to $\widetilde{\Q}$ is due to the presence of $\widetilde{\Q}$ in the parameterizations of both $\A$ and $\H$. The gradients corresponding to the parametrization of $\H$ can be derived by considering each frontal slice of $\H$ as
\begin{equation}
    (\H)_k = (\widetilde{\S})_k \widetilde{\Q}.
\end{equation}
Taking a small variation in $(\H)_k$, we can similarly arrive at gradients of $\overline{L}$ with respect to $(\widetilde{\S})_k$ and $\widetilde{\Q}$ as
\begin{equation}
    G_{(\widetilde{\S})_k} = G_{(\H)_k} \widetilde{\Q}^\top, \quad G_{\widetilde{\Q}}^{(\H)} = \sum_{k=1}^r (\widetilde{\S})_k^\top G_{(\H)_k},
\end{equation}
which then leads to the full gradient of $\overline{L}$ with $\widetilde{\Q}$,
\begin{equation}
    G_{\widetilde{\Q}} = (\widetilde{\K} - \widetilde{\MR})^\top G_{\A} + \sum_{k=1}^r (\widetilde{\S})_k^\top G_{(\H)_k},
\end{equation}
and \eqref{eq:gradQtilde} can be recovered by substituting the parameterizations of $\widetilde{\K}$ and $\widetilde{\MR}$ given by \eqref{eq:matrix_param}. The above expressions provide gradients of $\overline{L}$ with respect to the constrained tensors. To compute gradients with respect to the unconstrained parameters of \eqref{eq:optproblem} outlined in equations \eqref{eq:matrix_param} and \eqref{eq:pd_param}. For example, taking a small variation in $\widetilde{\K}$, the inner product with $G_{\widetilde{\K}}$ results in
\begin{align*}
    \inp*{G_{\widetilde{\K}}}{\delta \widetilde{\K}} &= \inp*{G_{\widetilde{\K}}}{\delta \K} - \inp*{G_{\widetilde{\K}}}{\delta \K^\top} \\
    &= \inp*{G_{\widetilde{\K}} - G_{\widetilde{\K}}^\top}{\delta \K},
\end{align*}
giving the gradient of $\overline{L}$ with respect to $\K$ as $G_{\widetilde{\K}} - G_{\widetilde{\K}}^\top$. After some substitution, \eqref{eq:gradJ} can be trivially recovered. The other gradients can be obtained similarly, and the proof is concluded.
\clearpage

\section{Derivation of Gradients for GasOpInf}
\label{ap:opinf_grads}
To begin, we rewrite \eqref{eq:opinf} by defining $\Z \in \R^{r \times S}$ where the columns of $\Z$ are given by all $S$ snapshots of $\bfz$, and $\U \in \R^{m \times S}$ where the columns of $\U$ are given by all $S$ snapshots of $\bfu$. This allows us to remove the summation over snapshots and write the objective function as
\begin{equation}
    J = \left \lVert \dot{\Z} - \A\Z - \H : \Z \Z^\top - \B \U \right \rVert^2_F.
\end{equation}
Defining the residual $\E = \dot{\Z} - \A \Z - \H:\Z\Z^\top - \B \U$, the objective function $J$ can be rewritten as the Euclidean tensor inner product
\begin{equation}
    J = \tr(\E^\top \E) = \inp*{\E}{\E}.
\end{equation}
By taking a small variation in $\A$, $\delta \A$, and holding everything else fixed, the variation in $\E$ is written as
\begin{equation}
    \delta \E = -(\delta \A) \Z,
\end{equation}
and the variation in $J$ for any $\delta \E$ is given as
\begin{equation}
    \delta J = \tr(\delta (\E^\top \E)) = \tr((\delta \E)^\top \E + \E^\top (\delta \E)) = \tr((\delta \E)^\top \E) + \tr(\E^\top (\delta \E))
\end{equation}
Using the properties of the trace operator, we rearrange the above expression to arrive at $\delta J = 2 \tr(\E^\top (\delta \E)) = 2 \inp*{\E}{\delta \E}$. Substituting $\delta \E = -(\delta \A) \Z$,
\begin{equation}
    \delta J = -2 \inp*{\E}{-(\delta \A) \Z} = -2 \tr(\E^\top (\delta \A)\Z) = -2 \tr(\Z \E^\top (\delta \A)) = \inp*{-2 \E \Z^\top}{\delta \A}
\end{equation}
Therefore, the gradient of the objective function with respect to $\A$, denoted here as $G_{\A}$, is given by
\begin{equation}
    G_{\A} = -2 \E \Z^\top.
\end{equation}
The gradient of $J$ with respect to the control $\B$ can be computed in a similar manner, resulting in
\begin{equation}
    G_{\B} = -2 \E \U^\top.
\end{equation}
To compute the gradient of $J$ with respect to the tensor $\H$, we look at the $j$-th column of the residual, $\E_j$, given by
\begin{equation}
    \E_j = \dot{\Z}_j - \A \Z_j - \H : \Z_j \Z_j^\top - \B \U_j,
\end{equation}
where the $i$-th component of the contraction $\H : \Z_j \Z_j^\top$ is given by
\begin{equation}
    (\H : \Z_j \Z_j^\top)_i = \sum_{p=1}^r \sum_{q=1}^r \H_{ipq} \Z_{pj} \Z_{qj}.
\end{equation}
The $i$-th component of the variation in $\E_j$ due to a variation in $\H$ is
\begin{equation}
    (\delta \E_j)_i = \sum_{p,q} (\delta \H)_{ipq} \Z_{pj} \Z_{qj}.
\end{equation}
Now that we have $(\delta \E_j)_i$, we can compute the variation in $J$ over all $j$ as
\begin{equation}
    \delta J = 2 \sum_{j=1}^S \E_j^\top \delta \E_j = -2 \sum_j \sum_{i,p,q} \E_{ij} (\delta \H)_{ipq} \Z_{pj} \Z_{qj}.
\end{equation}
Rearranging the right-hand side of the above expression, the components of the gradient of $J$ with respect to $\H$ can be isolated as
\begin{equation}
    (G_{\H})_{ipq} = -2 \sum_{j=1}^S \E_{ij} \Z_{pj} \Z_{qj},
\end{equation}
thus giving the gradients of the GasOpInf objective function with respect to the overall operators $\A$, $\H$, and $\B$. For the parametrized tensors, the same approach outlined in Appendix \ref{ap:proof} can be applied, as it was derived for gradients of a generalized objective function with respect to $\A$ and $\H$.
\clearpage

\section{Cavity flow results for $r=30$} \label{ap:cavity_r30}
Training was performed analogously to Section \ref{sec:cavity_training} with the same training dataset, but here considering the first 30 POD modes for the projection bases of POD-Galerkin, Operator Inference, and GasOpInf models, and as the initial condition for NiTROM and GasNiTROM. The Operator Inference and GasOpInf models were computed with a regularization parameter of $\lambda = \num{1e7}$ and $\lambda = \num{2e-2}$, respectively. Here, we optimize all models over the entire time window to show the accuracy of the model subject to a lower reduced-order dimension.

We compare GasNiTROM against all other models in the same fashion as Section \ref{sec:cavity_testing}. The training and testing errors for GasNiTROM, NiTROM, GasOpInf, Operator Inference, and the POD-Galerkin model (all with dimension $r=30$) are shown in Figures \ref{fig:ap_cavity_training_error} and \ref{fig:ap_cavity_testing_error}, respectively.

\begin{figure}
    \centering
    \begin{subfigure}[b]{0.4\textwidth}
        \centering
        \includegraphics[width=\textwidth]{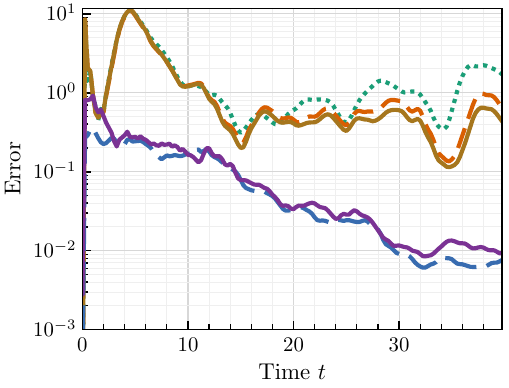}
        \caption{Training error}
        \label{fig:ap_cavity_training_error}
    \end{subfigure}\hspace{1.5cm}
    \begin{subfigure}[b]{0.4\textwidth}
        \centering
        \includegraphics[width=\textwidth]{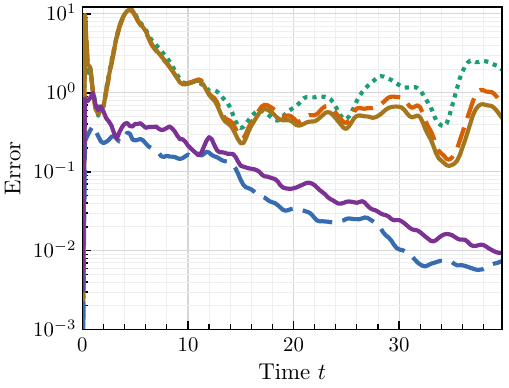}
        \caption{Testing error}
        \label{fig:ap_cavity_testing_error}
    \end{subfigure}
    \caption{Cavity flow ($r=30$) (a) training error from the 7 training impulse responses, and (b) testing error for the 25 unseen impulse responses. The error is defined in \eqref{eq:cavity_error}.}
    \label{fig:ap_cavity_error}
\end{figure}

As done previously, we use sinusoidal inputs of the form $w(t)=a \sin(kt)$, where $a \in \{0.1, 0.9\}$, initialized from the stable steady state. The results are shown in Figure \ref{fig:ap_cavity_forcing_energy}. In all cases, GasNiTROM performs equally as well or better than NiTROM. For $a=0.9$ and $k\in\{2,4\}$, the GasNiTROM prediction is not as reliable as lower forcing amplitudes, although it is still able to track the full-order model through the large initial transient and settle into a nearly correct steady-state. The other models are not able to do so, demonstrated by the trajectory generated from the NiTROM model blowing up for $k=4$.

\begin{figure}
    \centering
    \begin{subfigure}[b]{0.483\textwidth}
        \centering
        \includegraphics[width=\textwidth]{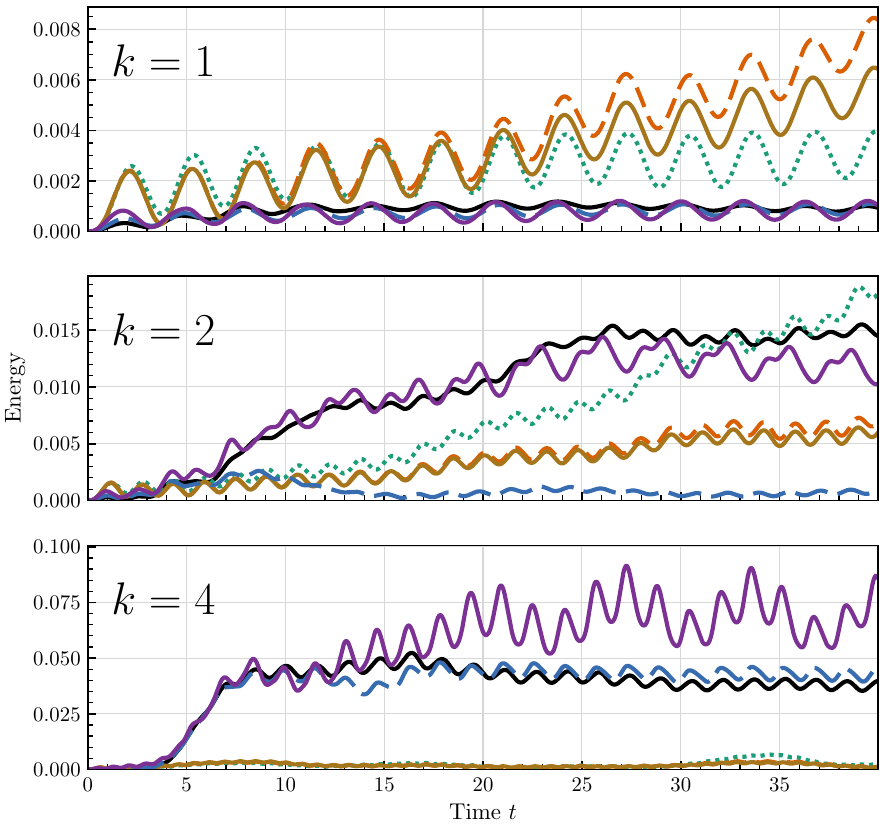}
        \caption{$a=0.1$}
        \label{fig:ap_cavity_forcing_0p1_energy}
    \end{subfigure}\hfill
    \begin{subfigure}[b]{0.477\textwidth}
        \centering
        \includegraphics[width=\textwidth]{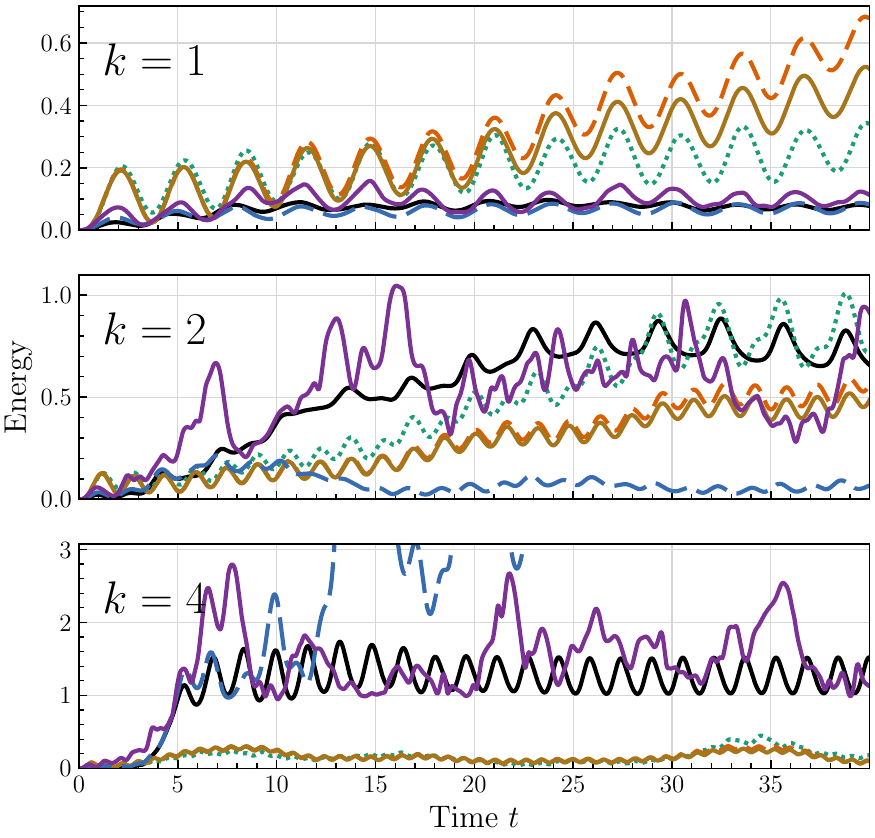}
        \caption{$a=0.9$}
        \label{fig:ap_cavity_forcing_097_energy}
    \end{subfigure}
    \caption{Cavity flow ($r=30$) time-response of the energy of the perturbations in response to sinusoidal inputs $w(t)=a\sin(kt)$ for an $a$ of (a) 0.1, (b) 0.9}
    \label{fig:ap_cavity_forcing_energy}
\end{figure}

We again show vorticity snapshots at $t=30$ from the response of sinusoidal forcing with $k=4$ and varying $a$ from $a=0.1$ to $a=0.9$. The $a=0.1$ case is shown in Figure \ref{fig:ap_cavity_forcing_snapshots_0p1}, and shows how Operator Inference, GasOpInf, and POD-Galerkin underestimate the magnitude of the vorticity and predict out of phase structures compared to the FOM. In Figure \ref{fig:ap_cavity_forcing_snapshots_0p9} the magnitude is increased to $a=0.9$, and the trajectory produced from NiTROM experiences blow-up, while all other models remain stable but uncorrelated with the full-order model.

\begin{figure}
    \centering
    \begin{subfigure}[b]{0.48\textwidth}
        \centering
        \includegraphics[width=\textwidth]{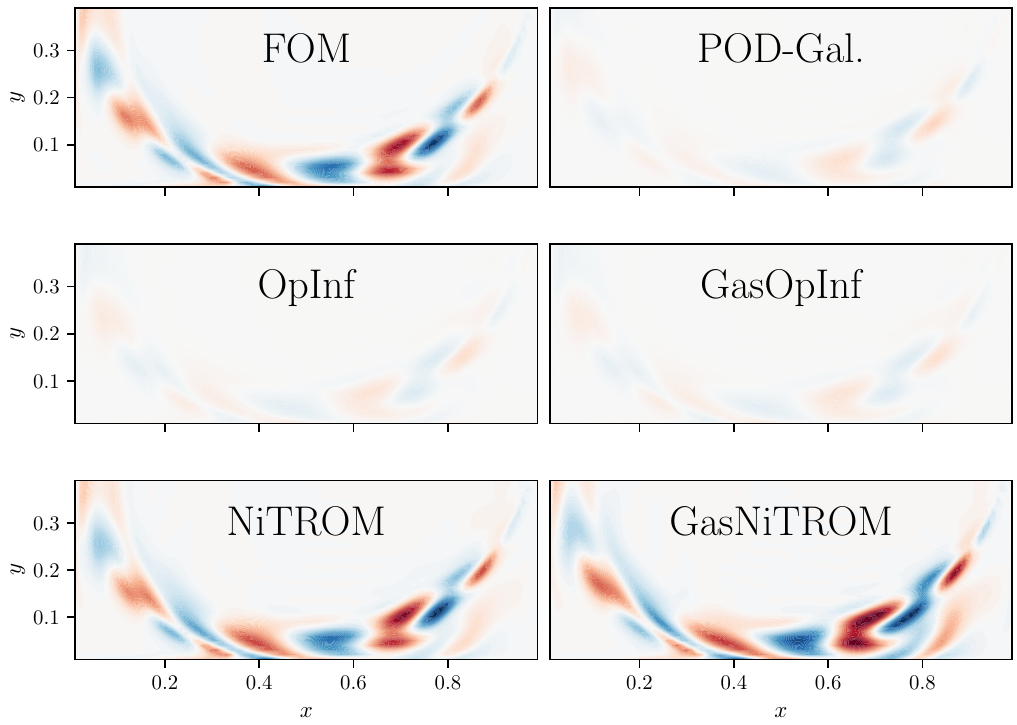}
        \caption{$a=0.1$, maximum vorticity 0.71}
        \label{fig:ap_cavity_forcing_snapshots_0p1}
    \end{subfigure}\hfill
    \begin{subfigure}[b]{0.48\textwidth}
        \centering
        \includegraphics[width=\textwidth]{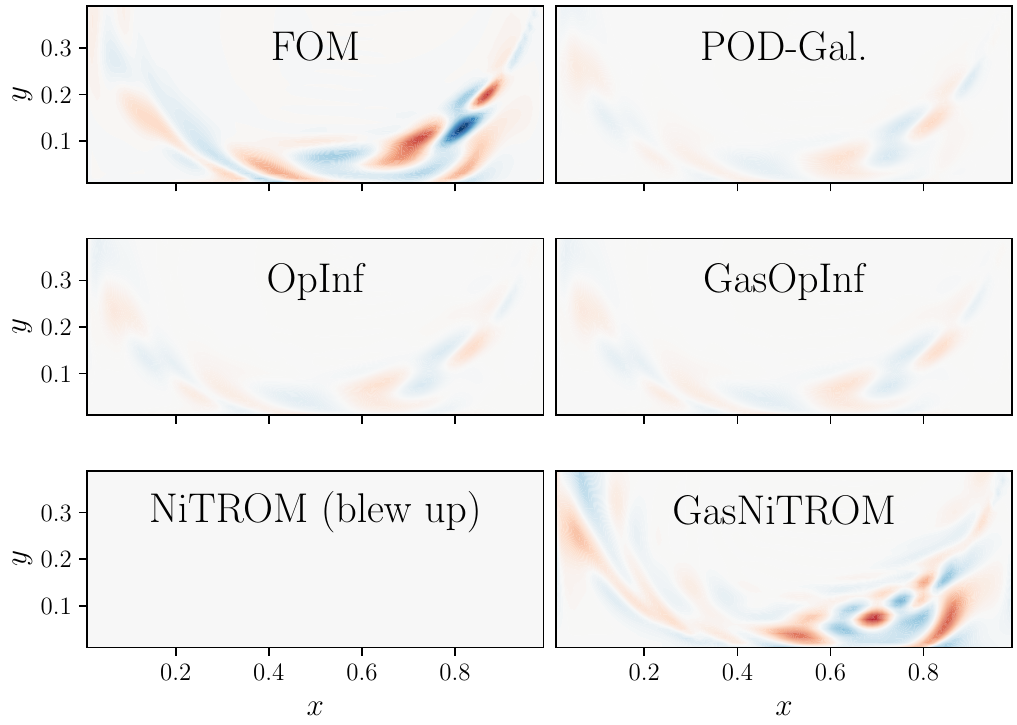}
        \caption{$a=0.9$, maximum vorticity 6.1}
        \label{fig:ap_cavity_forcing_snapshots_0p9}
    \end{subfigure}
    \caption{Cavity flow vorticity field at time $t=30$ given sinusoidal forcing with (a) $a=0.1$, (b) $a=0.9$, $k=4$ in Figure \ref{fig:ap_cavity_forcing_energy}. Red indicates positive vorticity with maximum value stated, blue indicates negative vorticity with minimum value of opposite sign of the maximum, and white is zero vorticity.}
    \label{fig:ap_cavity_forcing_snapshots}
\end{figure}

% \begin{figure}
%     \centering
%     \includegraphics[width=0.7\textwidth]{figures/cavity_30_forcing_0p1_k4_snapshot_all}
%     \caption{Cavity flow ($r=30$) vorticity field at time $t=30$ given sinusoidal forcing with $a=0.1$, $k=4$ in Figure \ref{fig:cavity_forcing_energy}. Red indicates positive vorticity with maximum value 0.71, blue indicates negative vorticity with minimum value -0.71, and white is zero vorticity.}
%     \label{fig:ap_cavity_forcing_snapshots_0p1}
% \end{figure}
% \begin{figure}
%     \centering
%     \includegraphics[width=0.7\textwidth]{figures/cavity_30_forcing_0p9_k4_snapshot_all}
%     \caption{Cavity flow ($r=30$) vorticity field at time $t=30$ given sinusoidal forcing with $a=0.7$, $k=4$ in Figure \ref{fig:cavity_forcing_energy}. Red indicates positive vorticity with maximum value 6.1, blue indicates negative vorticity with minimum value -6.1, and white is zero vorticity.}
%     \label{fig:ap_cavity_forcing_snapshots_0p9}
% \end{figure}

\end{document}